\documentclass[a4paper]{amsart}

\usepackage{amssymb,pstricks,amscd,epsfig}
\usepackage[utf8]{inputenc}
\usepackage[totalwidth=17cm,totalheight=24cm]{geometry}
\usepackage{graphicx}
\usepackage{mathrsfs}
\usepackage{mathtools}

\usepackage{amssymb,amscd}
\usepackage{amsmath,amsfonts,latexsym}
\usepackage{color}
\usepackage[english]{babel}
\usepackage{tikz}
\usetikzlibrary{matrix,arrows}
\usepackage[colorlinks,linkcolor=blue]{hyperref}
\usepackage{pgf,pgfplots}
\newcounter{notes}%

\newtheorem{cor}{Corollary}[section]
\newtheorem{theorem}[cor]{Theorem}
\newtheorem{prop}[cor]{Proposition}
\newtheorem{lemma}[cor]{Lemma}

\newtheorem{introthm}{Theorem}

\usepackage{amsthm}

\makeatletter
\newtheorem*{rep@theorem}{\rep@title}
\newcommand{\newreptheorem}[2]{%
\newenvironment{rep#1}[1]{%
 \def\rep@title{#2 \ref{##1}}%
 \begin{rep@theorem}}%
 {\end{rep@theorem}}}
\makeatother

\newreptheorem{theorem}{Theorem}

\theoremstyle{definition}
\newtheorem{defi}[cor]{Definition}

\newtheorem{question}[cor]{Question}

\newtheorem{remark}[cor]{Remark}
\theoremstyle{plain}

\newcommand{\cE}{{\mathcal E}}

\newcommand{\cT}{{\mathcal T}}

\newcommand{\cML}{{\mathcal M\mathcal L}}

\newcommand{\cQS}{{\mathcal Q\mathcal S}}
\newcommand{\C}{{\mathbb C}}
\newcommand{\DD}{{\mathbb D}}
\newcommand{\HH}{{\mathbb H}}

\newcommand{\N}{{\mathbb N}}

\newcommand{\R}{{\mathbb R}}

\newcommand{\Z}{{\mathbb Z}}

\newcommand{\lambdab}{\bar\lambda}

\newcommand{\RR}{\mathbb R}

\newcommand{\CP}{\mathbb{CP}}

\newcommand{\RP}{\mathbb{RP}}

\newcommand{\fonction}[5]{\begin{array}{cccl}           % Pour écrire une fonction en environnement display : le nom puis espaces de départ et d'arrrivée, puis argument et résultat
#1: & #2 & \longrightarrow & #3 \\
    & #4 & \longmapsto & #5 \end{array}}

\newcommand{\abs}[1]{\left\vert#1\right\vert}
\newcommand{\set}[1]{\left\{#1\right\}}
\newcommand{\crochet}[1]{\left[#1\right]}

\begin{document}

\newcommand{\bsut}{Bers Simultaneous Uniformization Theorem}
\newcommand{\re}{\mathrm{Re}}
\newcommand{\Jd}{\dot{J}}

\newcommand{\QS}{{\mathcal QS}}
\newcommand{\dwp}{d_{WP}}

\title{Bending laminations on convex hulls of anti-de Sitter quasicircles}

\author{Louis Merlin}
\thanks{Partially supported by FNR grant OPEN/16/11405402.}
\address{Department of mathematics, FSTM,
University of Luxembourg, 
Maison du nombre, 6 avenue de la Fonte,
L-4364 Esch-sur-Alzette, Luxembourg
}
\email{louis.merlin@uni.lu}

\author{Jean-Marc Schlenker}
\thanks{Partially supported by FNR grants INTER/ANR/15/11211745 and OPEN/16/11405402. The author also acknowledge support from U.S. National Science Foundation grants DMS-1107452, 1107263, 1107367 ``RNMS: GEometric structures And Representation varieties'' (the GEAR Network).}
\address{Department of mathematics, FSTM, 
University of Luxembourg, 
Maison du nombre, 6 avenue de la Fonte,
L-4364 Esch-sur-Alzette, Luxembourg
}
\email{jean-marc.schlenker@uni.lu}

\date{v1, \today}

% \begin{abstract}
% We show that, given two bounded measured laminations on $\mathbb{RP}^1\times\mathbb{RP}^1$ that strongly fill (to be defined), there exists an acausal quasicircle in $\partial AdS^3$ for which the boundary of its convex hull inside $AdS^3$ is a union of two surfaces pleated along the laminations. This extends a result of Bonahon-Otal for the quasi-Fuchsian case and translates into a result on earthquakes of the hyperbolic disk.
% \end{abstract}

\begin{abstract}
  Let $\lambda_-$ and $\lambda_+$ be two bounded measured laminations on the hyperbolic disk $\HH^2$, which ``strongly fill'' (definition below).
  We consider the left earthquakes along $\lambda_-$ and $\lambda_+$, considered as maps from the universal Teichm\"uller space $\cT$ to itself, and we prove that the composition of those left earthquakes has a fixed point. 
  The proof uses anti-de Sitter geometry. Given a quasi-symmetric homeomorphism $u:\RP^1\to \RP^1$, the boundary of the convex hull in $AdS^3$ of its graph in $\RP^1\times\RP^1\simeq \partial AdS^3$ is the disjoint union of two embedded copies of the hyperbolic plane, pleated along measured geodesic laminations. Our main result is that any pair of bounded measured laminations that ``strongly fill'' can be obtained in this manner. 
\end{abstract}

\maketitle

\tableofcontents

\section{Introduction and main results} %% JMS

%\subsection{The bending conjecture, from quasifuchsian manifolds to quasicircles}

\subsection{Quasi-symmetric homeomorphisms}

We denote by $\DD$ the unit disk in $\C$. Let $f:\DD\to \DD$ be a diffeomorphism. Its {\em conformal distorsion} at a point $x\in \DD^2$ is the smallest real $K\geq 1$ such that $cg_{Eucl}\leq f^*g_{Eucl}\leq Kcg_{Eucl}$ for some $c>0$, and $f$ is $K$-{\em quasi-conformal} if its conformal distorsion is at most $K$ everywhere. It is {\em quasi-conformal} if it is $K$-quasi-conformal for some $K\geq 1$. 

\begin{defi}
  A homeomorphism $u:\RP^1\to \RP^1$ is {\em quasi-symmetric} if it is the boundary value of a quasi-conformal diffeomorphism $f:\DD\to \DD$. (Here $\RP^1$ is identified with $\partial \DD$.) 
\end{defi}

We denote by $\QS$ the space of quasi-symmetric orientation-preserving homeomorphisms from $\mathbb{RP}^1$ to $\mathbb{RP}^1$.

The universal Teichmüller space $\mathcal{T}$ is the quotient of $\mathrm{Homeo}_{qs}(\mathbb{RP}^1)$ by PSL$_2(\mathbb{R})$ (acting on $\mathrm{Homeo}_{qs}(\mathbb{RP}^1)$ by post-composition):
\[\mathcal{T}=\mathrm{PSL}_2(\mathbb{R})\backslash \mathrm{Homeo}_{qs}(\mathbb{RP}^1).\]
It is indeed ``universal'' as it contains a copy of the Teichmüller space of a genus $g$ surface for any $g$, see {\em e.g.} \cite{bers:universal,gardiner-harvey}. 

\subsection{Measured laminations in the hyperbolic plane}

We now consider the hyperbolic plane $\HH^2$, which can be identified with $\DD$ using the Poincaré model. A {\em geodesic lamination} is a closed subset in $\HH^2$ which is a disjoint union of complete geodesics, and a {\em measured geodesic lamination} is a geodesic lamination equiped with a transverse measure, see e.g. \cite{FLP} and Section \ref{ssc:laminations}. We denote by $\cML$ the space of measured laminations on $\HH^2$.

A measured geodesic lamination $\lambda$ is {\em bounded} if there exists $C>0$ such that any geodesic segment of unit length has intersection at most $C$ with $\lambda$. We denote by $\cML_b$ the space of bounded measured geodesic laminations on $\HH^2$. 

A measured geodesic lamination $\lambda\in \cML$ defines a discontinuous map from $\HH^2$ to itself called a {\em left earthquake} along $\lambda$, see \cite{thurston-earthquakes,bonahon-extension}. In the simpler case when the support of $\lambda$ is discrete and each leaf has an atomic weight, the left (resp. right) earthquake along $\lambda$ corresponds to cutting $\HH^2$ along each leaf of $\lambda$, sliding the left (resp. right) side by a distance equal to the weight, and gluing back. We denote the left (resp. right) earthquake along $\lambda$ by $E^l_\lambda$ (resp. $E^r_\lambda$). (Note that the definition depends on the orientation chosen along $\lambda$, but the resulting map does not.)

Thurston \cite{thurston-earthquakes} proved that:
\begin{enumerate}
\item for any bounded measured lamination $\lambda$, the left earthquake $E^l_\lambda$ extends as a quasi-symmetric homeomorphism from $\partial_\infty \HH^2$ (identified with $\RP^1$) to itself,
\item any quasi-symmetric homeomorphism from $\partial_\infty \HH^2$ to itself is the boundary value of the left earthquake along a unique bounded measured lamination. 
\end{enumerate}

It is useful here to consider earthquakes as maps from the universal Teichm\"uller space to itself.

\begin{defi}
  We denote by $\cE^l:\cML_b\times \cQS\to \cQS$ the map defined as
  $$ \cE^l(\lambda)(u)=E^l(u_*\lambda)\circ u~, $$
  and similarly for $\cE^r$.
\end{defi}

Finally we need a notion of ``filling'' pair of measured bending laminations.

\begin{defi} \label{df:fill}
  Let $\lambda$ and  $\mu$ be two mesured laminations on  $\HH^2$. We say that $\lambda$ and $\mu$ \textit{strongly fill} if, for any $\varepsilon>0$, there exists $c>0$ such that, if $\gamma$ is a geodesic segment in $\mathbb{H}^2$ of length at least $c$, 
\[i(\gamma,\lambda)+i(\gamma,\mu)\geqslant \varepsilon~.\]
\end{defi}

For instance, if $\lambda_S$ and $\mu_S$ are two measured lamination on a closed surface $S$ that together fill, then their lifts $\lambda,\mu$ are bounded measured laminations on $\HH^2$ that strongly fill in our sense. %\marginnote{JM: true but add a reference (or remove the remark)?}

\subsection{Fixed points of compositions of earthquakes}

We can now give a first formulation of our main result.

\begin{introthm}\label{tm:earthquakes}
  Let $\lambda_-, \lambda_+\in \cML_b$ be two laminations that strongly fill. There exists a quasi-symmetric homeomorphism $u:\RP^1\to \RP^1$ such that $\cE^l(\lambda_l)(u)=\cE^r(\lambda_r)(u)$.
\end{introthm}

\begin{question} \label{q:unique}
  In this setting, is $u$ unique? 
\end{question}

Another way to state Theorem \ref{tm:earthquakes}, reminiscent of the main result in \cite{earthquakes}, is that if $\lambda_l, \lambda_r\in \cML$ strongly fill, then $\cE^l(\lambda_l)\circ\cE^l(\lambda_r):\cQS\to \cQS$ has a fixed point. Question \ref{q:unique} is equivalent to asking whether this fixed point is unique.

The proof of Theorem \ref{tm:earthquakes} can be found in Section \ref{sc:earthquakes}, where it is proved that it follows from Theorem \ref{tm:bending} below.

\subsection{The anti-de Sitter space and its boundary}

The proof of Theorem \ref{tm:earthquakes} uses anti-de Sitter geometry. Anti-de Sitter space is the Lorentzian cousin of hyperbolic space. In dimension $3$, it can be defined as the projectivisation of a quadric in the flat space $\R^{2,2}$ of signature $(2,2)$:
$$ AdS^3= \{ x\in \R^{2,2}~|~\langle x,x\rangle=-1\}/\{ \pm 1\} $$
with the induced metric. It is a Lorentzian space of constant curvature $-1$, homeomorphic to $\DD\times S^1$.

This space $AdS^3$ has a projective model analogous to the Klein model of $\HH^3$, where a ``hemisphere'' of $AdS^3$ is mapped to the interior of the quadric of equation $x^2+y^2=1+z^2$ in $\RR^3$. $AdS^3$ is in this manner equiped with a ``projective boundary'' $\partial AdS^3$. This projective boundary is naturally endowed with a conformal Lorentzian structure, analogous to the conformal structure on $\partial_\infty\HH^3$, see Section \ref{ssc:AdS}. In fact $\partial AdS^3$ can be identified with $\RP^1\times \RP^1$, with $\{ x\}\times \RP^1$ or $\RP^1\times \{ y\}$ corresponding to either
\begin{itemize}
\item the isotropic lines of the Lorentzian conformal structure on $\partial AdS^3$, or
\item the lines on the quadric of equation $x^2+y^2=z^2+1$ in $\R^3$. 
\end{itemize}

\subsection{Quasicircles in $\partial AdS^3$}

The following definition seems to be the natural analog in $AdS$ geometry of the notion of quasi-circles in $\CP^1$. Quasi-circles in $\partial AdS^3$ are ``space-like'' in a weak sense in $\partial AdS^3$ equiped with its conformal Lorentzian structure -- the proper notion here being that they are acausal meridians, with a limited regularity. %, see Section \ref{ssc:AdS}.%\marginnote{L: vérifier}

\begin{defi} \label{df:quasicircle}
  A {\em quasicircle} in $\partial AdS^3$ is the graph of a quasi-symmetric homeomorphism from $\RP^1$ to $\RP^1$, in the identification of $\partial AdS^3$ with $\RP^1\times \RP^1$.
\end{defi}

There are a number of reasons to believe that, albeit being quite different in appearance, this is actually the ``correct'' analog of quasicircles in $\CP^1$. For instance, Mess \cite{mess,mess-notes} showed that globally hyperbolic 3-dimensional $AdS$ spacetimes are analogous in deep ways to quasifuchsian hyperbolic manifolds, and their limit set in $\partial AdS^3$ is a quasicircle in the sense of Definition \ref{df:quasicircle}, just as the limit set of a quasifuchsian hyperbolic manifold is a quasicircle in $\CP^1$. Another analogy appears in \cite{convexhull}, where quasicircles appear in both the hyperbolic and $AdS$ setting as ideal boundaries of pleated surfaces and constant curvature surfaces. 

We will also use a natural notion of {\em parameterized} quasicircle, already seen in \cite{convexhull}.

\begin{defi}
  A {\em parameterized quasicircle} in $\partial AdS^3$ is a map $u:\RP^1\to \partial AdS^3$ such that, under the identification of $\partial AdS^3$, the composition on the left of $u$ with either the left or the right projection is quasi-symmetric.
\end{defi}

It follows from this definition that the image of a parameterized quasicircle is a quasicircle according to Definition \ref{df:quasicircle}.

\subsection{Bending laminations on the boundary of the convex hull}

Given a quasi-circle $C\subset \partial AdS^3$, its convex hull in the projective model of $AdS^3$ is a convex subset of $AdS^3$ with a boundary composed of two space-like surfaces $\partial_\pm CH(C)$. The geometry of those surfaces was analysed by Mess \cite{mess} in analogy with the hyperbolic convex hull of quasi-circles in $\CP^1$. In both cases the boundaries are the disjoint unions of two pleated hyperbolic planes, that is, isometrically embedded copies of the hyperbolic plane pleated along measured lamination $\lambda_-, \lambda_+$.

\begin{introthm} \label{tm:bending}
Let $\lambda_-, \lambda_+\in \cML_b$ two bounded measured laminations that strongly fill. There exists a parameterized quasicircle $u:\mathbb{RP}^1\to \partial AdS^3$ such that the measured bending laminations on the upper and lower boundary components of $CH(u(\RP^1))$ are $u_*(\lambda_+)$ and $u_*(\lambda_-)$, respectively.     
\end{introthm}

Theorem \ref{tm:bending} might not be optimal, in the sense that the conditions that $\lambda_-, \lambda_+$ fill strongly might not be necessary. In fact the only {\em necessary} conditions that we know on $\lambda_-, \lambda_+$ is that they must fill in a much weaker sense (each complete geodesic in $\HH^2$ has positive intersection with either $\lambda_-$ or $\lambda_+$, see Section \ref{ssc:necessary}). In Section \ref{ssc:example}, we show by an example that this ``weak'' filling condition is not sufficient.

Question \ref{q:unique} is equivalent to asking whether the parameterized quasicircle $u$ in Theorem \ref{tm:bending} is unique (up to post-composition by a isometry of $AdS^3$). 

\subsection{Related results}

The results presented here are related to a number of recent results in hyperbolic or $AdS$ geometry.

\subsubsection*{Quasifuchsian $AdS$ spacetimes}

This corresponds to the case where the quasicircle $C$ is invariant under the action of a surface group, which acts on the domain of dependence of $C$ % $CH(C)$
with quotient a globally hyperbolic compact maximal $AdS$ spacetime. The past and future boundary components of $CH(C)$ are then also invariants, and their quotients by the surface group action are the future and past boundary components of the convex core of the quotient spacetime.
The measured laminations $\lambda_\pm$ can be considered as measured laminations on a closed surface $S$. Theorems \ref{tm:earthquakes} and \ref{tm:bending} then reduce to the main results in \cite{earthquakes}.

Here, too, uniqueness remains elusive.

\subsubsection*{Quasifuchsian hyperbolic manifolds}

The situation is similar for quasifuchsian manifolds. In this case, it was proved by Bonahon and Otal \cite{bonahon-otal} that any pair of measured laminations on a closed surface, that fill and have no closed leaf with weight larger than $\pi$, can be realized as the measured bending lamination on the boundary of the convex core of a quasifuchsian manifold. Uniqueness, however, is only known for laminations whose support is a multicurve.

The analog of Theorem \ref{tm:bending} in the hyperbolic context -- but without group action -- is not known.

\subsubsection*{$K$-surfaces in $AdS^3$}

Theorem \ref{tm:bending} can also be considered when the boundaries of the convex hull of $C$ is replaced by a pair of convex surfaces of constant curvature $K<-1$ in $AdS^3$. The bending measure is then replaced by the {\em third fundamental forms} of the $K$-surfaces, and the analog of Theorem \ref{tm:bending} for those $K$-surfaces is proved in \cite{convexhull}. There is an analog of Theorem \ref{tm:earthquakes} associated to those $K$-surfaces, where earthquakes are replaced by {\em landslides} as introduced in \cite{cyclic,cyclic2}.

\subsection{Examples and limitations}

We will see in Section \ref{sc:examples} that although our main results are presumably not optimal, the precise statement of an optimal result is not quite as simple as one could imagine.

\subsection*{Acknowledgement}

The second-named author would like to thank Francesco Bonsante for useful conversations related to the content of this paper.

%%%%%%%%%%%%%%%%%%%%%%%%%%%%%%%%%%%%%%%%%%%%%%%%%%%%%%%%%%%%%%%

%% explain AdS^3, projective boundary, notion of quasicircle and parameterized quasicircle.

% Let $f:X\rightarrow Y$ be a diffeomorphism between two Riemann surfaces (not necessarily compact). We define the map $\mu$ as the solution of the equation $\frac{\partial f}{\partial \bar{z}}=\mu(z)\frac{\partial f}{\partial z}$ (hence $\mu$ is the Beltrami differential of $f$). We say that $f$ is $k$-quasi-conformal if $k(f)=\frac{1+\norm{\mu}_\infty}{1-\norm{\mu}_\infty} $ is less than $k$.

% \begin{defi}
% We say that a homeomorphism $u:\mathbb{RP}^1\rightarrow \mathbb{RP}^1$ is $k$-quasi-symmetric if it is the boundary extension of a $k$-quasi-conformal map $f:\mathbb{H}^2\rightarrow\mathbb{H}^2$ \footnote{A quasi-conformal homeomorphism does extend to an homeomorphism of $\mathbb{RP}^1$ \note{fletcher markovic 07}}.
% \end{defi}

%% notion of quasi-symm homeo.

% \begin{defi}

% \end{defi}

%\subsection{Bending laminations}

%% introduce bending lamination on boundary of convex hull.

%% explain that we consider measured laminations either on the disk or as measure on RP1xRP1, etc.

% We will denote by $ \cML$ the space of bounded measured laminations. Note that the notion of bounded measured lamination is invariant under quasi-symmetric homeomorphisms of $\mathbb{RP}^1$.\note{Need to add ref.}

%% denote by $ \cML$ the space of bounded measured laminations.

%% expression in terms of composition of earthquakes because inverse

 % intro -- JM

\section{Backbground material.} %% LM

\subsection{Cross-ratios}

In the sequel, we will also use a characterization of quasi-symmetric homeomorphisms in terms of cross-ratios. In order to state it, we denote the cross-ratio of a 4-tuple of points $(a,b,c,d)$ in $\left(\mathbb{R}\cup\set{\infty}\right)^4$ by $cr(a,b,c,d)$:
\[cr(a;b;c;d)=\frac{(c-a)(d-b)}{(b-a)(d-c)}.\]
The map $cr$ is invariant under the diagonal action of the Möbius group and thus determines a well-defined map
\[cr : \left( \mathbb{RP}^1\right)^4\rightarrow \mathbb{RP}^1\] when we identify $\mathbb{R}\cup\set{\infty}$ with $\mathbb{RP}^1$.
We say that $(a,b,c,d)$ are in symmetric position if they are the images of $(-1,0,1,\infty)$ by a M\"obius transformation (for the diagonal action on $\left(\mathbb{RP}^1\right)^4$). Equivalently $cr(a;b;c;d)=-1$.

\begin{lemma}[\cite{fletchermarkovic}]
The homeomorphism $u:\mathbb{RP}^1\rightarrow\mathbb{RP}^1$ is $k$-quasi-symmetric if and only if there exists $k'$ such that, for any symmetric 4-tuple of points $(a,b,c,d)$, then
\[-k\leqslant cr(u(a);u(b);u(c);u(d))\leqslant -\frac{1}{k'}\]
The constant $k$ goes to infinity if and only of $k'$ does.
\end{lemma}

\subsection{Measured laminations.}
\label{ssc:laminations}

%% .\marginnote{JM: add: measured laminations as measures on RP1xRP1. L: Faut-il faire plus détailé ?}

We defined in the introduction a lamination $\lambda$ as a closed set which is the disjoint union of complete geodesics in $\mathbb{H}^2$. A transverse invariant measure associated to a lamination $\lambda$ is a non-negative Radon measure defined on each embedded differentiable arc $\gamma$ which is transverse to $\lambda$, and such that, if there exists a homotopy sending $\gamma$ to $\gamma'$ while respecting $\lambda$, then the measure on $\gamma$ is the same as the measure on $\gamma'$. We say that the measured geodesic lamination has full support if the support of the transverse measure is exactly $\lambda$. For the rest of the paper, we always assume that the measured geodesic laminations have full support.

In $\mathbb{H}^2$, a (unparametrized) complete geodesic is characterized by the set of its 2 endpoints in $\partial_\infty\mathbb{H}^2$. Under the identification $\partial_\infty\mathbb{H}^2\simeq \mathbb{RP}^1$, a geodesic lamination is then a closed subset of
\[\left(\mathbb{RP}^1\times\mathbb{RP}^1\setminus\Delta\right)\slash \mathbb{Z}_2,\]
where $\Delta$ is the diagonal set in the product  $\mathbb{RP}^1\times\mathbb{RP}^1$ and $ \mathbb{Z}_2$ acts by switching the endpoints. As a consequence, one can also see measured laminations as measures on $\mathbb{RP}^1\times\mathbb{RP}^1\setminus \Delta$, with the choice of a section from $\left(\mathbb{RP}^1\times\mathbb{RP}^1\setminus\Delta\right)\slash \mathbb{Z}_2$ that does not need to be specified. This allows to define a topology on the space of measured laminations: the weak-$\star$ topology on measures on $\mathbb{RP}^1\times\mathbb{RP}^1\setminus \Delta$.

%\marginnote{JM: pourquoi cette section? L: Par exemple pour le lemme qui suit.}

\begin{lemma} \label{lem:reciprocalbounds}
Let $\lambda_+, \lambda_-$ be any two bounded laminations. Then there exists a constant $C$ such that, for any 4-tuple of pairwise distinct points $a,b,c,d\in\mathbb{RP}^1$, one has
\[\min \left\lbrace \lambda_+\left(\crochet{a,b}\times\crochet{c,d}\right)+\lambda_-\left(\crochet{a,b}\times\crochet{c,d}\right),\lambda_+\left(\crochet{a,c}\times\crochet{b,d}\right)+\lambda_-\left(\crochet{a,c}\times\crochet{b,d}\right)\right\rbrace\leqslant C.\]
\end{lemma}

\begin{proof}
Denote by $\gamma_{ab}$, $\gamma_{cd}$, $\gamma_{ac}$ and $\gamma_{bd}$ the four geodesics in $\mathbb{H}^2$ connecting the points $ab$, $cd$, $ac$ and $bd$. Since the point $a,b,c,d$ are distinct, there exist two uniquely determined geodesic segments $h$ and $k$, perpendicular to $\gamma_{ab}$ and $\gamma_{cd}$ and $\gamma_{ac}$ and $\gamma_{bd}$ respectively. Any leaf of $\lambda_\pm$ connecting $[a,b]$ to $[c,d]$ must intersect $h$, and similarly for $k$. Since $\lambda_+$ and $\lambda_-$ are bounded, the result will follow from the fact that either the length of $h$ or the length of $k$ is bounded. But it is a classical fact in hyperbolic geometry \cite[theorem 2.3.1 (i)]{buser:spectra} that
\[\sinh\left(\frac{L(h)}{2}\right)\sinh\left(\frac{L(k)}{2}\right)=1\]
which means that $h$ and $k$ cannot both be of length greater than $2\mbox{argsinh}(1)$. 
\end{proof}

%\marginnote{JM: removed lemma on strongly filling laminations that was not obvious and does not seem to be used? L: Ok, let's check again when it's all over}
% \begin{lemma}
% The two laminations $\lambda_-$ and $\lambda_+$ fill strongly if and only if there exists $C>0$ such that for all 4-tuple of points $(a,b,c,d)$ in symmetric position, $\lambda_-([a,b]\times [c,d])+\lambda_+([a,b]\times [c,d])\geq C$.
% \end{lemma}

\subsection{On the geometry of the  $AdS^3$-space.}
\label{ssc:AdS}

We recall in this section some key properties of the 3-dimensional $AdS$ space. The original source for most of the points described here is \cite[Section 7]{mess}, see also \cite{mess-notes}. More details can be found in \cite{bonsante-seppi:anti} or in the background sections of \cite{barbot-merigot,earthquakes,maximal,idealpolyhedra}.

%%.\marginnote{Trouver une bonne réf}

\subsubsection*{The Lie group model.}

We consider the group PSL$_2(\mathbb{R})$ with its Killing form $\kappa$. We recall that $\kappa$ is defined on the Lie algebra $\mathfrak{sl}_2(\mathbb{R})$ of the group PSL$_2(\mathbb{R})$ by, for $u,v\in\mathfrak{sl}_2(\mathbb{R})$,
\[\kappa(u,v)=\mathrm{tr}(ad(u)\circ ad(v)).\]
The Killing form is $Ad$-invariant and so defines a pseudo-Riemannian metric on the whole group PSL$_2(\mathbb{R})$ (still denoted $\kappa$). Its signature is $(2,1)$.

We define the Anti de Sitter space of dimension 3 ($AdS^3$) as the group PSL$_2(\mathbb{R})$ together with the Lorentzian metric $g_{AdS^3}=\frac{1}{8}\kappa.$ The normalizing constant $\frac{1}{8}$ is made so that the Lorentzian curvature of $g_{AdS^3}$ is precisely $-1$.

Since the metric $g_{AdS^3}$ is undefinite, the tangent vectors have a type that we call:
\begin{itemize}
\item space-like if its $g_{AdS^3}$ squared norm is positive,
\item time-like if its $g_{AdS^3}$ squared norm is negative and
\item light-like if its $g_{AdS^3}$ squared norm vanishes.
\end{itemize}

In any tangent space, the set of light-like vectors forms a cone, and the set of time-like vectors has two connected components: the future-pointing time-like vectors and the past-pointing time-like vectors. Even though the space $AdS^3$ is not simply connected, the choice of future-pointing and past-pointing time-like vectors can be done consistently. We refer to this choice as the time-orientation of $AdS^3$. It is also oriented by the choice of a tangent basis of the form $(u,v, \crochet{u,v})$ for any two space-like vectors $u,v$.

The time-preserving, orientation-preserving isometry group of $AdS^3$ satisfies
\[\mathrm{Isom}_0(AdS^3)= SO_0(2,2)=SO_0(2,1)\times SO_0(2,1)=
  \mathrm{PSL}_2(\mathbb{R})\times \mathrm{PSL}_2(\mathbb{R}). \]
One way to recover this identification is by remarking that PSL$_2(\R)$ acts on itself by left and right multiplications. When PSL$_2(\R)$ equipped with its Killing metric is identified with $AdS^3$, the action of PSL$_2(\R)\times$ PSL$_2(\R)$ by left and right multiplication is isometric, because the Killing form of PSL$_2(\R)$ is left and right invariant. The corresponding morphism from PSL$_2(\R)\times $PSL$_2(\R)$ to $\mathrm{Isom}_0(AdS^3)$ turns out to be an isomorphism.

% With the fact that PSL$_2(\mathbb{R})$ is the orientation preserving isometry group of $\mathbb{H}^2$, we also get that
% \[\partial AdS^3\simeq \partial_\infty\mathbb{H}^2\times\partial_\infty\mathbb{H}^2.\]

\subsubsection*{The projective model.}

We now construct another avatar of $AdS^3$, called the projective model.

On the space $\mathcal{M}_2(\mathbb{R})$, we denote by $q$ the quadratic form $-\det $. The associated symmetric bilinear form $\left\langle\cdot,\cdot\right\rangle$ has signature $(2,2)$. It is easy to see that the restriction of $\left\langle\cdot,\cdot\right\rangle$ to SL$_2(\mathbb{R})$ corresponds to (the double cover of) the metric $g_{AdS^3}$, so that we can identify $AdS^3$ with
\[\mathrm{PSL}_2(\mathbb{R})=\set{A\in\mathcal{M}_2(\mathbb{R})\;\;q(A)=-1}\slash\set{\pm 1},\]
endowed with the metric descending from $\left\langle\cdot,\cdot\right\rangle$.
We choose the system of coordinates in $\mathcal{M}_2(\mathbb{R})$ so that it is diffeomorphic to $\mathbb{R}^4$ through the diffeomorphism
\[\fonction{\varphi}{\mathbb{R}^4}{\mathcal{M}_2(\mathbb{R})}{(x_1,x_2,x_3,x_4)}{\begin{pmatrix}
x_1-x_3 & -x_2+x_4\\
x_2 + x_4 & x_1+x_3
\end{pmatrix}}.\]
It is easy to check that, in this coordinate system, we have
\[q(A)=-x_1^2-x_2^2+x_3^2+x_4^2,\]
so that $AdS^3$ identifies with the subset
$$ \Omega=\{ [x_1,x_2,x_3,x_4]\in \RP^3~|~ -x_1^2-x_2^2+x_3^2+x_4^2<0\} $$
in $\mathbb{RP}^3$. The metric $g_{AdS^3}$ is compatible with projective geometry, in the sense that geodesics of $AdS^3$ are projective lines and the isometry group $\mathrm{Isom}AdS^3$ identifies with the subgroup of PGL$_4(\mathbb{R})$ with preserves $q$, i.e the group PO$(2,2)$. See \cite{fillastreseppi}.

\subsubsection*{Boundary of $AdS^3$}

As an open set of $\mathbb{RP}^3$, $AdS^3$ has a natural compactification and the boundary is the 2-dimensional Einstein space Ein$^{1,1}$. Referring to the projective model above, Ein$^{1,1}$ is defined as the projectivization of isotropic lines for $q$. The restriction of $\left\langle\cdot,\cdot\right\rangle$ has signature $(1,1)$ and the action of $\mathrm{Isom}AdS^3$ extends to Ein$^{1,1}$ and acts by $q_{|\mathrm{Ein}^{1,1}}$-conformal transformations. We usually write Ein$^{1,1}=\partial AdS^3$.

The Lie group model of $AdS^3$ yields the identification Ein$^{1,1}\simeq \partial_\infty\mathbb{H}^2\times\partial_\infty\mathbb{H}^2$ whereas the projective model yields Ein$^{1,1}\simeq \mathbb{RP}^1\times\mathbb{RP}^1$. Either way Ein$^{1,1}$ is then a topological torus and $AdS^3$ is a solid torus. The projective model of $AdS^3$ allows to endow Ein$^{1,1}$ with a double ruling by projective lines of the form $\mathbb{RP}^1\times\set{\star}$ (left ruling) or of the form $\set{\star}\times\mathbb{RP}^1$ (right ruling). Those lines are the null-lines for $q$. Note that there is also a time-orientation of Ein$^{1,1}$ coming from the time-orientation of $AdS^3$.

Since $\mathrm{Isom}(AdS^3)$ acts projectively on the projective model $\Omega$, it acts projectively on $\partial \Omega$. So the action of $\mathrm{Isom}_0(AdS^3)$ sends lines to lines, and the identity component $\mathrm{Isom}_0(AdS^3)$ acts separately on each family of lines. Each of those families of lines is equipped with a real projective structure -- coming from the intersection with any line of the other family -- and the action of $\mathrm{Isom}_0(AdS^3)$ on each family of lines is projective. This defines a morphism from $\mathrm{Isom}_0(AdS^3)$ to PSL$_2(\R)\times $PSL$_2(\R)$ which can be shown to be an isomorphism. 

%% geometrie, courbure, modèle projectif, bord projectif

%% structure lorentz conforme du bord et structure causale du bord

\subsubsection*{The left and right projections}

Let $\Pi_0$ be a fixed totally geodesic space-like plane in $AdS^3$. Its boundary $\partial\Pi_0\subset \partial AdS^3$ is a circle which intersects exactly once every leaf of the left and of the right foliation of $\partial AdS^3$. The product decomposition $\partial AdS^3=\RP^1\times \RP^1$ can therefore be used to project $\partial AdS^3$ to $\partial \Pi_0$ along the leaves of the left or of the right projection, leading to projection maps $\pi_L,\pi_R:\partial AdS^3\to \partial \Pi_0$. Replacing $\Pi_0$ by another totally geodesic space-like leads to composing $\pi_L, \pi_R$ on the left with a Möbius transformation.

Given another totally geodesic space-like plane $\Pi\subset AdS^3$, the restrictions of $\pi_L,\pi_R$ to $\partial \Pi$ are Möbius transformations, which have a unique extension, that we call $\pi_{L,\Pi}, \pi_{R,\Pi}$, to an isometry from $\Pi$ to $\Pi_0$.

If now $\Sigma$ is a space-like surface in $AdS^3$, we can define left and right projections, still denoted by $\pi_L, \pi_R$, from $\Sigma$ to $\Pi_0$, in the following manner. For each point $x\in \Sigma$, let $\Pi_x$ be the totally geodesic space-like plane tangent to $\Sigma$ at $x$, and we define $\pi_L(x)=\pi_{L,\Pi_x}(x),\pi_R(x)=\pi_{R,\Pi_x}(x)$. It was already noted by Mess \cite{mess} that if $\Sigma$ is a pleated surface, then this construction defines a left (resp. right) earthquake along the measured bending lamination from $\Sigma$ equipped with its induced (hyperbolic) metric to $\Pi_0$. 

\subsubsection*{Acausal meridians in $\partial AdS^3$}

Let again $\Sigma\subset AdS^3$ be a space-like surface, that we now consider to be extrinsically complete. Its boundary is then a curve $\partial \Sigma\subset \partial AdS^3$ which is weakly space-like in the conformal Lorentzian structure of $\partial AdS^3$, in the sense that no two points can be connected by a short time-like segment. 

In the identification of $\partial AdS^3$ with $\RP^1\times \RP^1$, acausal meridians are graphs of functions from $\RP^1$ to $\RP^1$. Those functions might however not be continuous, when the acausal meridian contains a light-like segment. Among the acausal meridians, it is natural to consider those which are graphs of more regular functions. As mentioned below, we call {\em quasicircle} an acausal meridian which is the graph of a quasi-symmetric homeomorphism.

\subsubsection*{Convex hulls of acausal meridians}

It was already proved in \cite{mess} that an acausal curve $C\subset \partial AdS^3$ is always disjoint from a certain totally geodesic space-like plane. As a consequence, it is fully contained in an affine chart of $\RP^3$, and therefore in the boundary of a projective model of $AdS^3$ in $\RR^3$.

One can therefore define its convex hull in $AdS^3$, in an affine manner, and the resulting subset does not depend on the affine chart that is chosen (assuming of course that it contains $C$). The complement of $C$ in the boundary of this convex hull is then the union of two surfaces which are everywhere space-like or light-like, one future-oriented, the other past-oriented.

If $C$ is a quasicircle, then each of those connected component of the complement of $C$ in the boundary of $CH(C)$ is space-like. It is moreover isometric, with its induced metric, to the hyperbolic plane, and is pleated along a bounded measured foliation, see \cite{mess}.

\subsection{The Rhombus}\label{subsec:rhombus}

We now describe a very special acausal curve in $\partial AdS^3$ which plays a key role in a number of compactness questions in $AdS$ geometry.%\marginnote{Still need to update this part.}

\begin{defi}[The Rhombus]
Consider four points $a<a'$ and $b<b'$ in $\mathbb{RP}^1$ \footnote{The two points $a$ and $a'$ live on the same lightlike line in $\partial AdS^3$, namely $\mathbb{RP}^1\times\left\lbrace\star\right\rbrace$; this line is oriented by the time. Same for $b$ and $b'$.}. The Rhombus is a curve inside $\partial AdS^3$ which connects the points $(a,b),(a',b),(a',b')$ and $(a,b')$ in this order by lightlike lines (i.e of the form $\mathbb{RP}^1\times\left\lbrace\star\right\rbrace$ or $\left\lbrace\star\right\rbrace\times\mathbb{RP}^1$).

We also call Rhombus the convex hull inside $AdS^3$ of such a curve. It is a tetrahedron, with two past oriented and two future oriented triangular faces. It has six edges, four of which being contained in $\partial AdS^3$ and light-like and the two remaining ones space-like lines in $AdS^3$ that we call the \textit{axis}. %%\marginnote{Recently added to complete lemma 4.5}
\end{defi}

\subsection{The width of acausal meridians}

Another, related notion, also important for compactness issues in $AdS$ geometry, is the {\em width} of an acausal curve.

\begin{defi}
The width of an acausal quasi-circle $C$ is the supremum of the time distance between $\partial_- CH(C)$ and $\partial_+ CH(C)$.
\end{defi}

As an example, it can be checked that the width of a rhombus is $\pi/2$, and any point in the past axis is connected to any point in the future axis by a time-like geodesic segment of length exactly $\pi/2$.

The following proposition appeared in \cite{maximal}, and can also be found as \cite[Proposition 6.8]{convexhull}. %\marginnote{JM: we could add that if $\pi/2$ then after normalization by a sequence of isometries then the acausal curve converges to a rhombus.}

\begin{lemma} 
Let $C$ be an acausal curve in $\partial AdS^3$. Its width is at most $\pi/2$. If $C$ is a quasicircle, then the width of $C$ is less than $\frac{\pi}{2}$. Conversely if the width of $C$ is exactly $\frac{\pi}{2}$, then $C$ is not a quasi-circle and one of the two cases occur.
\begin{enumerate}
\item The distance $\frac{\pi}{2}$ is achieved, in which case, $C$ is a rhombus.
\item The distance $\frac{\pi}{2}$ is not achieved and there exists a sequence of isometries $\varphi_n\in\mathrm{Isom}_0(AdS^3)$ such that $(\varphi_n (C))_{n\in \N}$ converges to a rhombus in the Hausdorff topology on compact subsets of $AdS^3$. 
\end{enumerate}
\end{lemma}

%%%%%%%%%%%%%%%%%%%%%%%%% left over %%%%%%%%%%%%%%%%%%%%%%%%%%%%%%%%%%%%%%%%
% We now define a quasi-circle as the image in $\Ein^{1,1}$ of a quasi-symmetric homeomorphism. More precisely, let $C$ be a continuous curve in $\Ein^{1,1}$. We say that $C$ is an acausal meridian if it bounds a disk in $AdS^3$ and, for every $p\in C$, there exists a neighborhood $U$ of $p$ in $\Ein^{1,1}$, such that every point $x\in U\cap C$ can be connected to $p$ by a space-like curve.

% \begin{lemma}[\note{mess}]
% Let $C$ be an acausal meridian. Then there exists a homeomorphism $u:\mathbb{RP}^1\rightarrow\mathbb{RP}^1$ such that $C$ is the graph of $u$.
% \end{lemma}

% We now say that $C$ is a $k$-quasi-circle if the homoemorphism given in the previous lemma is $k$-quasi-symmetric.

% \note{We also have to present the characterization of quasi-circles as being the quasi-symmetric gluing map as in the most recent BDMS, theorem D.}

% Finally we define the universal Teichmüller space. Denote the group of quasi-symmetric homeomorphisms of $\mathbb{RP}^1$ by $\mathrm{Homeo}_{qs}(\mathbb{RP}^1)$. 
% %% def of universal Teichm\"uller space by quotient
 % background material -- LM

\section{Ideal polyhedra and approximation of laminations.} %% JMS

In this section, we show how to approximate a pair of measured laminations that strongly fill in $\HH^2$ by a sequence of pairs of ``polyhedral'' laminations, that is, measured laminations on ideal polygons with support contained in a union of disjoint diagonals. Those polyhedral laminations will be constructed so that they satisfy a result on the dihedral of ideal polyhedra in $AdS^3$, so as to provide a sequence of polyhedra with dihedral angles converging, in a proper sense, to the pair of laminations. We will then show that this sequence of ideal polyhedra converges to the convex hull of an acausal curve, as needed for the proof of Theorem \ref{tm:bending}

\subsection{Ideal polyhedra with prescribed dihedral angles}

We first define what we mean by a ``polyhedral'' measured lamination on $\HH^2$.

\begin{defi}
  A measured lamination on $\HH^2$ is {\em polyhedral} if its support is the disjoint union of a finite set of complete geodesics.
\end{defi}

This notion is related to that of {\em ideal polyhedron} in $AdS^3$. It is quite simple to define ideal polyhedra in the projective model of $AdS^3$, see Section \ref{ssc:AdS}: it is a polyhedron with all vertices on the boundary $\partial AdS^3$, but such that the complement of the vertices is contained in $AdS^3$. It follows from this definition that all edges and all faces must be space-like, see \cite{idealpolyhedra}. (There is also a similar notion of hyperideal polyhedron, see \cite{hyperideal}).

\begin{defi}
  Let $P$ be an ideal polyhedron in $AdS^3$, and let $P_-$ and $P_+$ be its past and future boundary components.
  Each of those two pleated ideal polygons defines a polyhedral measured lamination. We denote these two laminations by $\lambda_-(P)$ and $\lambda_+(P)$ respectively.
  We denote by $V(P)$ the $n$-tuple of vertices of $P$, cyclically oriented in the order in which they occur on any acausal meridian containing them. 
\end{defi}

By construction, $\lambda_\pm(P)$ have a support which is a disjoint union of diagonals, that is, complete geodesics connecting two vertices. 

\begin{theorem}\label{thm:approx}
Let $\lambda_-$, $\lambda_+\in \cML$ be a pair of laminations that strongly fill.
There exists a sequence $(P_n)_{n\in \N}$ of ideal polyhedra and a sequence $u_n$ of functions from $V(P_n)$ to $\RP^1$ such that $((u_{n})_*(\lambda_-(P_n)))_{n\in \N}$ and $((u_{n})_*(\lambda_+(P_n)))_{n\in \N}$ converge to $\lambda_-$ and $\lambda_+$, respectively.
\end{theorem}

The convergence considered here is in the sense of weak-$*$ convergence on measures in $\RP^1\times \RP^1\setminus \Delta$, see section \ref{ssc:laminations}.

In the proof of Theorem \ref{thm:approx}, the main tool to exhibit ideal polyhedra is \cite[Theorem 1.4]{idealpolyhedra} which we recall below.

Let $\Gamma$ be a weighted graph, embedded in the 2-dimensional sphere. For an edge $e$ of this graph, we denote by $\theta(e)$ its weight. To $\Gamma$, we associate the dual graph $\Gamma^*$. Given an edge $e$ of $\Gamma$, we denote by $e^*$ the dual edge.
We say that a polyhedron $P(\Gamma)$ in $AdS^3$ is the realization of $\Gamma$ if there is a bijection between the edges of $\Gamma$ and the edges of $P(\Gamma)$ taking the weight of $e\in E(\Gamma)$ to the exterior dihedral angle of the corresponding edge in $P(\Gamma)$. \cite[Theorem 1.4]{idealpolyhedra} states that $\Gamma$ can be realized if and only if the following criteria are fulfilled.
\begin{enumerate}
\item The graph $\Gamma$ has a Hamiltonian path.
\item For any edge $e$ of the Hamiltonian path, we have $\theta(e)<0$, and $\theta(e)>0$ otherwise.
\item If $e^*_1,\cdots,e^*_k$ bound a face of $\Gamma^*$, then
\[\sum_{i=1}^k \theta(e^*_i)=0.\]
\item If $e^*_1,\cdots,e^*_k$ is a simple continuous path in $\Gamma^*$, which does not bound a face of $\Gamma^*$, then
\[\sum_{i=1}^k \theta(e^*_i)>0.\]
\end{enumerate}
Moreover the polyhedron $P(\Gamma)$ realizing $\Gamma$ enjoys the following property: the vertices can all be connected by a polygonal curve composed of the edges separating two faces for which the exterior normals have opposite time orientation. Those edges are the ones coming from the Hamiltonian path.
We refer to this polygonal curve as the {\em equator} of $P(\Gamma)$. 

In light of this result, our goal becomes to construct a weighted graph satisfying properties (1) to (4) above.

\subsection{Approximation of a pair of laminations}

Let us start by an elementary construction.

We fix $n\in\mathbb{N}$, and we consider the disk $D(o,n)$ of radius $n$ centered at a fixed point $o\in\mathbb{H}^2$ . We remove from $\lambda_-$ and $\lambda_+$ the leaves which does not intersect $D(o,n)$. We denote those truncated laminations by $\lambda_-^n$ and $\lambda_+^n$; our aim is to approximate them by polyhedral laminations. Making an arbitrarily small perturbation of $\lambda_-$ and $\lambda_+$, we assume that those truncated laminations are \textit{rational}: they consist in disjoint union of isolated leaves and the measure is a finite sum of Dirac masses. Moreover, we can assume (and will use below) that the weight of each leaf is also a rational number. The total weight of those truncated lamination is finite, it equals the sum of the intersections of the circle $C(o,n)$ with the laminations $\lambda_+$ and $\lambda_-$. We denote this total weight by $\Lambda_n$.

We now arrange $k$ points on the circle $\mathbb{S}^1=\partial_\infty\mathbb{H}^2$ according to the following procedure. The number $k$ will be chosen later as a function of $n$. We use the Poincaré disk model of $\mathbb{H}^2$. We first set $a_1=i$.

Assuming that $a_j$ has been constructed, we construct $a_{j+1}$ in such a way that the leaves of $\lambda_-^n$ and $\lambda_+^n$ limiting in the interval $\left(a_ja_{j+1}\right)$ have a total weight at most $\frac{\Lambda_n}{k}$ and that the leaves of $\lambda_-^n$ and $\lambda_+^n$ limiting in the interval $\left(a_ja_{j+1}\right]$ have a total weight at least $\frac{\Lambda_n}{k}$.
The vertex $a_{j+1}$ is then split in two vertices, one still named $a_{j+1}$ and another one say $a'_{j+1}$ which is very close to $a_{j+1}$ in the counterclockwise direction, and any weight of leaves ending at $a_{j+1}$ in excess of $\frac{\Lambda_n}{k}$ is moved to $a'_{j+1}$. We refer to this process as leaf-splitting. %\marginnote{JM: on pourrait ajouter un dessin de ce splitting?}

The endproduct of this construction is $k$-tuple of points $a_1, \cdots, a_k$, arising in this cyclic order, such that the total weight of the leaves ending in $(a_j,a_{j+1}]$ is always exactly $\Lambda_n/k$.

We now consider the following graph $\Gamma_0$:
\begin{itemize}
\item $\Gamma_0$ has a vertex $v_j$ for each of the $k$ points $a_j$.
\item The edges between $v_j$ and $v_{j'}$ correspond to the leaves of $\lambda_-^n$ and $\lambda_+^n$ from the interval $\left(a_ja_{j+1}\right]$ to $\left(a_{j'}a_{j'+1}\right]$, if there is any such leaf.
\item The adjacent vertices $v_j$ and $v_{j+1}$ are related by an edge of weight $-\frac{\Lambda_n}{2k}$.
\end{itemize}
Roughly speaking the graph $\Gamma_0$ is the union of $\lambda^n_-$ and $\lambda^n_+$ to which we pulled back the endpoints of the leaves in the interval $\left[a_{j}a_{j+1}\right)$ to the single vertex $v_j$.
Remark that, since the laminations $\lambda^n_-$ and $\lambda^n_+$ are assumed to be rational, we can arrange that each of the points $a_j$ and hence $v_j$ is the endpoint of some leaf in $\lambda^n_-$ or $\lambda^n_+$.

\begin{lemma}\label{lem:laminations_graph}
There is an infinite number of integers $k$ so that the graph $\Gamma_0$ is exactly the union of $\lambda^n_-$ and $\lambda^n_+$, maybe with leaves splitted, to which we have added the equatorial edges.
\end{lemma}

\begin{proof}
Recall that we have assumed that the lamination $\lambda^n_-$ and $\lambda_n^+$ are rational. We then denote the weights of the leaves by $\frac{p_1}{q_1},\cdots,\frac{p_m}{q_m}$. We now choose $k$ to be an integer multiple of $\frac{q_1 q_2\cdots q_m }{\Lambda_n}$, say $k=\alpha\frac{q_1\cdots q_m}{\Lambda_n}$. It follows that the weight of any leaf is a integer multiple of $\frac{\Lambda_n}{k}$ (indeed $\frac{p_i}{q_i}=\alpha p_i\cdot q_1\cdots \hat{q_i}\cdots q_m\frac{\lambda_n}{k}$). The leaf-splitting procedure gives rise to $2\alpha p_i\cdot q_1\cdots \hat{q_i}\cdots q_m$ vertices in $\Gamma_0$, close to the endpoints of the leaf of weight $\frac{p_i}{q_i}$ and edges, all of which having the weight $\frac{\Lambda_n}{k}$ between the vertices, and close to the original leaf.
%\marginnote{L: J'ai ré-écrit cette preuve, est-ce qu'on comprend mieux ?}
\end{proof}

The graph $\Gamma_0$ does not yet encode the 1-skeleton and dihedral angle of the polyhedron we are looking for. We need to construct a slight modification of $\Gamma_0$ to which we can apply \cite[Theorem 1.4]{idealpolyhedra}. Nevertheless, $\Gamma_0$ is the starting point of our construction: it obviously satisfies properties (1), (2) and (3). Throughought the modifications of $\Gamma_0$, it will be easy to see that those three properties persist.

The ``slight modification'' will be understood in the following sense.

\begin{defi}
  We consider the set of weighted graphs $\Gamma=(V;E;\theta)$ in $\mathbb{H}^2$ such that:
  \begin{itemize}
  \item the vertices are on $\partial_\infty\mathbb{H}^2$,
  \item the set of edges is bipartite $E=E_+\cup E_-$,
  \item  both $(E_+,\theta_{|E_+})$ and $(E_-,\theta_{|E_-})$ are measured geodesic laminations.
  \end{itemize}
  We endow the set of such graphs with the topology coming from the topology on the space of measured laminations.
\end{defi}

We denote by $\lambda^{k,n}_-$ and $\lambda^{k,n}_+$ the laminations associated to $\Gamma_0$.

We can now modify $\Gamma_0$ to satisfy the requirements of \cite[Theorem 1.4]{idealpolyhedra} through the following three steps.

\vspace{0.5cm}

\paragraph{{\bf 1st step: coloring the vertices.}}

The first part of the modification for the weighted graph $\Gamma$ is to notice that its vertices
can be separated in two groups, depending on whether they are endpoints of an edge in $E_+$ of in $E_-$.

\begin{lemma}\label{lem:endpoints}
For any $n\in\mathbb{N}$, there is a number $\delta(n)>0$ with the following property. Let $a_+\in\mathbb{S}^1$ be the endpoint of a geodesic $F_+\in\lambda^n_+$ and $a_-\in\mathbb{S}^1$ be the endpoint of a geodesic $F_-\in\lambda^n_-$. Then
\[d(a_-,a_+)>\delta(n)\]
(where $d$ denotes the visual distance on $\partial_\infty \HH^3$ associated to the center $o$).
\end{lemma}

The proof will use the following simple lemma.

\begin{lemma} \label{lem:parallel}
  Let $\lambda_-,\lambda_+$ be two bounded measured geodesic lamination on the hyperbolic plane that strongly fill. There exists $\varepsilon>0$ and $L>0$ such that if $\gamma_+$ and $\gamma_-$ are geodesic segments in the support of $\lambda_+$ and $\lambda_-$, respectively, of length at least $L$, then one of the endpoints of $\gamma_+$ is at distance at least $\varepsilon$ from $\gamma_-$. 
\end{lemma}

\begin{proof}
  For clarity we consider $\gamma_-$ and $\gamma_+$ as oriented, and call $\gamma_{\pm,0}$ and $\gamma_{\pm,1}$ their endpoints. We assume by contradiction that $d(\gamma_{-,0},\gamma_{+,0})\leq \varepsilon$ and that $d(\gamma_{-,1},\gamma_{+,1})\leq \varepsilon$.

  Any leaf of $\lambda_+$ intersecting $\gamma_-$ must intersect either $[\gamma_{-,0},\gamma_{+,0}]$ or $[\gamma_{-,1},\gamma_{+,1}]$. However the strict filling hypothesis imply that the intersection of $\gamma_-$ with $\lambda_+$ is bounded from below by a constant which goes to infinity with $L$. So the sum of the intersections with $\lambda_+$ of $[\gamma_{-,0},\gamma_{+,0}]$ and of $[\gamma_{-,1},\gamma_{+,1}]$ is bounded from below by a constant going to infinity as $L\to \infty$. This however contradicts the fact that $\lambda_+$ is bounded, if $L$ is large and $\varepsilon$ is small.
\end{proof}

\begin{cor} \label{cr:parallel}
  For each $n>0$, there exists $\varepsilon_n>0$ as follows. Let $F_-$ and $F_+$ be leaves of $\lambda_-$ and $\lambda_+$, respectively, with endpoints $a_-$ and $a_+$ on $\partial \HH^2$, and let $b_-, b_+$ be the first intersection of $F_-, F_+$ with $C(o,n)$, when starting from $a_-, a_+$. Then either the visual metric from $o$ between $a_-$ and $a_+$ is at least $\varepsilon_n$, or the visual metric from $o$ between $b_-$ and $b_+$ is at least $\varepsilon_n$.
\end{cor}

\begin{proof}
  This clearly follows from Lemma \ref{lem:parallel}, since otherwise, for $\varepsilon_n$ small enough, the segments of $F_-$ and $F_+$ of length $L$ starting from their first intersections with $C(o,n)$ towards the outside of $D(o,n)$ would contradict that lemma.
\end{proof}

\begin{proof}[Proof of Lemma \ref{lem:endpoints}]
  We consider $a_-=F_-(+\infty)$ and $a_+=F_+(+\infty)$ as in the statement of Lemma \ref{lem:endpoints}, at distance less than $\varepsilon_n$. It follows by Corollary \ref{cr:parallel} that the points $b_-$ and $b_+$ which are the first intersections of $F_-$ and $F_+$ with $C(o,n)$ are at distance at least $\varepsilon_n$. In fact, with the same type of argument, we can consider a rescaled visual metric on $C(0,n)$ and we may as well assume that the points $b_-$ and $b_+$ are at distance at least $C\varepsilon$ for a constant $C>1$ which will be adjusted later.

  We consider two cases. 

  First case: the segments of $F_-$ and $F_+$ between $a_-$ (resp. $a_+$) and $C(o,n)$ intersect at a point $p$. Since $d_o(b_-,b_+)\geq Cd_o(a_-, a_+)$, the hyperbolic distance between $D(o,n)$ and $p$ is bounded from below by a function $F(n)$ of $n$ which is bounded away from 0 by a constant depending on $C$. Let $\gamma$ be a geodesic segment going from $p$ to a point of $C(o,n)$, staying within the triangle boundedy by $C(o,n)$ and by the segments of $F_-$ and $F_+$ between $p$ and $C(o,n)$. $\gamma$ has length at least $F(n)$ and therefore, as $n\to \infty$ and by choosing $C$ big enough, must have large intersection with either $\lambda_-$ or $\lambda_+$ (because $\lambda_-$ and $\lambda_+$ strongly fill). But any leave of $\lambda_-$ or $\lambda_+$ intersecting $\gamma$ must intersect $C(o,n)$, a contradiction. See figure \ref{fig:notation1}.

\definecolor{uuuuuu}{rgb}{0.26666666666666666,0.26666666666666666,0.26666666666666666}
\definecolor{ffqqqq}{rgb}{1,0,0}
\begin{figure}[!h]
\begin{tikzpicture}[line cap=round,line join=round,>=triangle 45,x=1cm,y=1cm]
\clip(-4.3,-3.78) rectangle (8.84,6.3);
\draw [shift={(1.5973181254381705,8.1258075823761)},line width=0.5pt]  plot[domain=4.029956078586286:5.384183684640062,variable=\t]({1*9.001858067944205*cos(\t r)+0*9.001858067944205*sin(\t r)},{0*9.001858067944205*cos(\t r)+1*9.001858067944205*sin(\t r)});
\draw [shift={(1.6503294117647058,7.604176470588235)},line width=0.8pt]  plot[domain=3.982907677116938:5.554881166812981,variable=\t]({1*7.007424392046069*cos(\t r)+0*7.007424392046069*sin(\t r)},{0*7.007424392046069*cos(\t r)+1*7.007424392046069*sin(\t r)});
\draw [line width=3.6pt,color=ffqqqq] (1.08,0.62)-- (1.06,-0.86);
\draw (-2.92,-0.04) node[anchor=north west] {$ \partial_\infty\mathbb{H}^2 $};
\draw (-3.9,2.4) node[anchor=north west] {$ B(o,n) $};
\draw (-3.9,3.4) node[anchor=north west] {$ F_- $};
\draw (3.5,2.6) node[anchor=north west] {$ F_+ $};
\draw (0.70,0.78) node[anchor=north west] {$ \gamma $};
\draw (0.36,-0.86) node[anchor=north west] {$ a_+ $};
\draw (1.5,-0.94) node[anchor=north west] {$ a_- $};
\draw [shift={(-5.593748101465419,-4.05670102767756)},line width=2pt]  plot[domain=0.4225681039466382:1.3427547327428608,variable=\t]({1*7.757535873271658*cos(\t r)+0*7.757535873271658*sin(\t r)},{0*7.757535873271658*cos(\t r)+1*7.757535873271658*sin(\t r)});
\draw [shift={(7.267323507617761,-3.444333807712094)},line width=2pt]  plot[domain=1.7059915715599123:2.761271550453879,variable=\t]({1*7.028468340494026*cos(\t r)+0*7.028468340494026*sin(\t r)},{0*7.028468340494026*cos(\t r)+1*7.028468340494026*sin(\t r)});
\draw (0.12,0.84) node[anchor=north west] {$b_-$};
\draw (1.54,0.78) node[anchor=north west] {$b_+$};
\begin{scriptsize}
\draw [fill=uuuuuu] (1.082344930721114,-0.10586002218428667) circle (2pt);
\draw[color=uuuuuu] (1.3,0.00) node {$p$};
\end{scriptsize}
\end{tikzpicture}
\caption{Lemma \ref{lem:endpoints} Case 1}
\label{fig:notation1}
\end{figure}
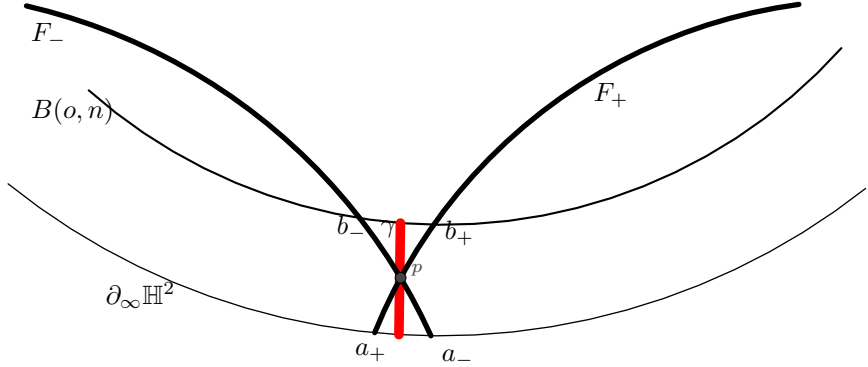

  Second case: those segments do not intersect. We then apply exactly the same argument, but with $p$ replaced by the intersection between the segments $[a_-,b_+]$ and $[a_+,b_-]$, see figure \ref{fig:notation2}.
\end{proof}

  % We then consider the geodesic arcs joining $F_+\cap C(o,n)$ to $a_-$ and $F_-\cap C(o,n)$ to $a_+$. Those two arcs intersect at a point $a$. We also consider a geodesic arc $\gamma$ whose initial point is on $C(o,n)$ between $F_-$ and $F_+$ and passing through $a$. The endpoint of this arc is necessarily between $a_-$ and $a_+$ (see figure \ref{fig:notation}).

%The geodesic $\gamma$ has infinite hyperbolic length and thus
%\[i(\gamma,\lambda_-)+i(\gamma,\lambda_+)=+\infty.\]
%Without loss of generality, we may assume that $i(\gamma,\lambda_+)=+\infty$.

%But every leaf of $\lambda_+$ meeting $\gamma$ between $a$ and $C(o,n)$ must enter the ball $B(o,n)$ and then belongs to $\lambda^n_+$. Moreover, if $d(a_-,a_+)$ goes to 0, then the length of the piece of $\gamma$ between $C(o,n)$ and $a$ goes to infinity ($a$ converges to the boundary). Hence $d(a_-,a_+)\to 0$ implies that $i(\gamma,\lambda^n_+)\to \infty$ which is a contradiction with the fact that the total mass of $\lambda^n_+$ is finite.
%\end{proof}

\definecolor{xdxdff}{rgb}{0.49,0.49,1}
\definecolor{ffqqqq}{rgb}{1,0,0}
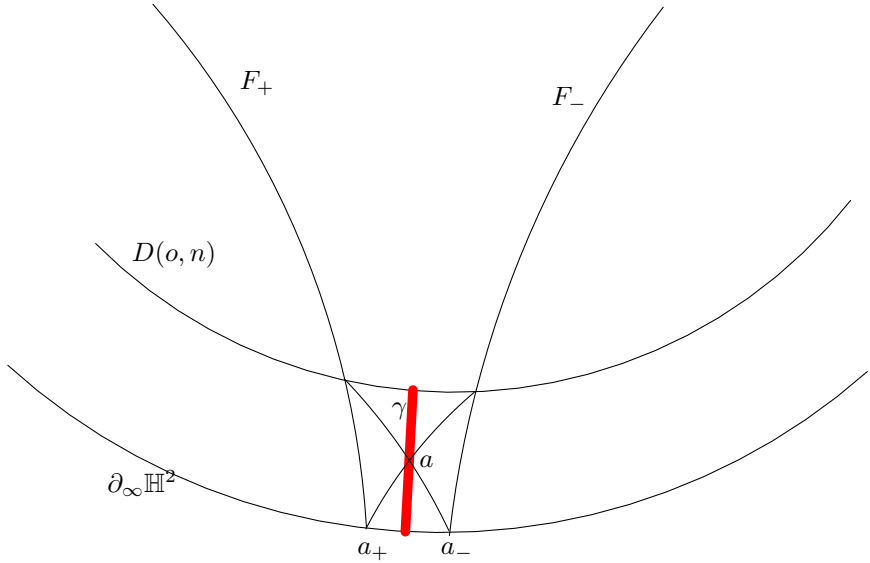
\begin{figure}[!h]
\begin{tikzpicture}[line cap=round,line join=round,>=triangle 45,x=1.0cm,y=1.0cm]
\clip(-4.3,-2.7) rectangle (7.4,6.3);
\draw [shift={(1.59,7.38)}] plot[domain=3.97:5.44,variable=\t]({1*8.43*cos(\t r)+0*8.43*sin(\t r)},{0*8.43*cos(\t r)+1*8.43*sin(\t r)});
\draw [shift={(1.76,7.43)}] plot[domain=3.92:5.62,variable=\t]({1*6.62*cos(\t r)+0*6.62*sin(\t r)},{0*6.62*cos(\t r)+1*6.62*sin(\t r)});
\draw [shift={(-11.09,-1.68)}] plot[domain=0.06:0.71,variable=\t]({1*11.72*cos(\t r)+0*11.72*sin(\t r)},{0*11.72*cos(\t r)+1*11.72*sin(\t r)});
\draw [shift={(15.77,-2.67)}] plot[domain=2.49:3.03,variable=\t]({1*14.16*cos(\t r)+0*14.16*sin(\t r)},{0*14.16*cos(\t r)+1*14.16*sin(\t r)});
\draw [line width=3.6pt,color=ffqqqq] (1.22,0.84)-- (1.12,-1.04);
\draw [shift={(5.91,-3.71)}] plot[domain=2.28:2.67,variable=\t]({1*5.96*cos(\t r)+0*5.96*sin(\t r)},{0*5.96*cos(\t r)+1*5.96*sin(\t r)});
\draw [shift={(-5.23,-4.28)}] plot[domain=0.44:0.76,variable=\t]({1*7.65*cos(\t r)+0*7.65*sin(\t r)},{0*7.65*cos(\t r)+1*7.65*sin(\t r)});
\draw (-2.92,-0.1) node[anchor=north west] {$ \partial_\infty\mathbb{H}^2 $};
\draw (-2.6,2.94) node[anchor=north west] {$ D(o,n) $};
\draw (-1.18,5.2) node[anchor=north west] {$ F_+ $};
\draw (2.92,4.98) node[anchor=north west] {$ F_- $};
\draw (0.82,0.8) node[anchor=north west] {$ \gamma $};
\draw (0.7,-1.3) node {$a_+$};
\draw (1.8,-1.3) node {$a_-$};
%\fill [color=xdxdff] (1.17,-0.1) circle (1.5pt);
\draw (1.4,-0.1) node {$a$};
\end{tikzpicture}
\caption{Lemma \ref{lem:endpoints} Case 2}
\label{fig:notation2}
\end{figure}

By construction, all vertices of $\lambda^{k,n}_+$ (resp. $\lambda^{k,n}_-$) are arbitrarily close, for $k$ large enough, from the endpoint of a leave of $\lambda^n_+$ (resp. $\lambda^n_-$). It follows that for $k$ large enough, the statement of Lemma \ref{lem:endpoints} also applies (with a slightly different function $\delta$) to the endpoints of $\lambda^{k,n}_+$ and $\lambda^{k,n}_-$.%\marginnote{JM: ajouté.}

Since any vertex $v_j$ of $\Gamma_0$ is the endpoint of a leaf in either $\lambda^{k,n}_-$ or $\lambda^{k,n}_+$, but not both according to the previous argument, it follows that each vertex has a well defined type ($+$ or $-$) depending on which type of leaf it terminates.

\vspace{0.5cm}

\paragraph{{\bf 2nd step: A graph $\Gamma_1$ which satisfies the weaker condition (4')}}

Property (4) reduces %% \marginnote{JM: ok but we need to explain why.}
to the following statement: let $\gamma$ be a geodesic path in $\mathbb{H}^2$ connecting the interval $\left(v_1v_2\right)$ to $\left(v_{p}v_{p+1}\right)$. Then the total weight of the edges of $\lambda_-^n$ and $\lambda_+^n$ it crosses is bigger than $\frac{\Lambda_n}{k}$. (The reduction comes from the fact that any closed path in $\Gamma^*$ intersecting the equator in two points can be ``straightened'' into the double cover of a geodesic. The weight of the intersection of this closed path with the equator is $-\frac{\Lambda_n}{k}$, so condition (4) is satisfied if and only if the sum of the intersections with all the geodesics in $\lambda_-^n$ and $\lambda_+^n$ is larger than $\frac{\Lambda_n}{k}$.)

The first modification of $\Gamma$ is done according to the following lemma

\begin{lemma}\label{lem:4'}
There exists a weighted graph $\Gamma_1$ arbitrarily close to $\Gamma_0$ which satisfies properties (1), (2), (3) and the following weakening of property (4):
\newline
(4') Let $\gamma$ be a geodesic path in $\mathbb{H}^2$. Then
\[i(\gamma,\lambda^{k,n}_-)+i(\gamma,\lambda^{k,n}_+)>0.\]
\end{lemma}

\begin{proof}
Note that two cases are easily understood: if $\gamma$ does not enter the disk $D(o,n)$ and joins the two intervals $(v_p,v_{p+1})$ to $(v_{p',}v_{p'+1})$, then $\gamma$ must cross every edge emanating between $v_{p+1}$ and $v_{p'}$, with total weight $|p'-p|\Lambda_n/k$. On the opposite, if the length of the intersection of $\gamma$ with the disk $D(o,n)$ is large enough, then the conclusion follows from the strong filling hypothesis. Nevertheless, it may happen that, for some $p$ and $p'$, a geodesic joining $(v_p,v_{p+1})$ to $(v_{p',}v_{p'+1})$ does not cross either $\gamma^{k,n}_-$ nor $\gamma^{k,n}_+$. In this case, we wish to add an edge between $v_p$ and $v_{p'}$ of small weight. We can indeed choose a lamination to which we add an edge since $\gamma$ cannot be close to both $\lambda^{k,n}_-$ and $\lambda_+^{k,n}$ (Lemma \ref{lem:endpoints}). %%\marginnote{L: Is it clear enough ?}
\end{proof}

We still denote by $\lambda^{k,n}_-$ and $\lambda^{k,n}_+$ the geodesic laminations associated to the modified graph $\Gamma_1$.
\vspace{0.5cm}

\paragraph{{\bf 3rd step: From (4') to (4)}} The last step is to prove that property (4') in fact implies property (4) for an arbitrarily close weighted graph. 

\begin{lemma}\label{lem:splitting}
Let $\Gamma_1$ be the graph constructed in step 2. We can find a graph $\Gamma$ arbitrarily close to $\Gamma_1$ satisfying properties (1) to (4).
\end{lemma}

\begin{proof}
For any pair of vertices $p$ and $p'$, we know that the geodesic joining $(v_p,v_{p+1})$ to $(v_{p',}v_{p'+1})$ has positive mass. Denote my $m=m(k)$ the smallest intersection:
\[m=\inf_{p,p'}\left\lbrace i(\gamma,\lambda^{k,n}_-)+i(\gamma,\lambda^{k,n}_+)\right\rbrace\]
where, for each $p,p'$, $\gamma$ is a geodesic joining $(v_p,v_{p+1})$ to $(v_{p',}v_{p'+1})$. (Note that the intersection number does not depend on which geodesic is chosen with endpoint in those intervals.)

We now split the set of vertices according to the following procedure. Any vertex $v_p$ is replaced by the set $v_p^{i_1},\cdots,v_p^{i_p}$ on $\mathbb{S}^1$ so that $v_p^{i_1},\cdots,v_p^{i_p}$ are cyclically ordered and arbitrarily close. We replace the equatorial edges by edges, for any $j$, from $v_p^{i_j}$ to $v_p^{i_{j+1}}$ of equal weight. We choose $i_p$ so that the weight of any edge $(v_p^{i_j},v_p^{i_{j+1}})$ is less that $m/2$.

The resulting graph satisfies condition (4) and is arbitrarily close to $\Gamma_1$ since the $v_{i_j}$ are arbitrarily close to the former $v_p$.
\end{proof}

\begin{proof}[Proof of Theorem \ref{thm:approx}]
For each integers $n$, we consider the graph $\Gamma_0$, with $k$ large enough so that $\Gamma_0$ satisfies the hypothesis of Lemma \ref{lem:laminations_graph}. We modify it once with the process described in Lemma \ref{lem:4'}, adding leaves of weights $\frac{1}{n}$ so that we get a graph satisfying properties (1) -- (4'). Then we add vertices as described in lemma \ref{lem:splitting}. The resulting graph satisfies properties (1) -- (4) and has a number of vertices $K=K(n)$. We can now apply \cite[Theorem 1.4]{idealpolyhedra} for each integer $n\in\mathbb{N}$ we get an ideal polyhedron $P_n$ in $AdS^3$ with $K$ vertices.

The fact that $\lambda_+(P_n)$ converge to $\lambda_+$ and $\lambda_-(P_n)$ converge to $\lambda_-$ as $n\to \infty$ follows directly from the definition of the topology on $\cML$ and Lemma \ref{lem:laminations_graph}. %% and the fact that rational laminations are dense in $\cML$.
\end{proof}

We now associate to those truncated laminations, a discrete approximation of the quasi-circle we are looking for. Indeed, let $P_n$ be the ideal polyhedra we constructed in the previous theorem. Its vertices form a subset $V_n\subset \partial AdS^3$, which are naturally ordered since each can be connected to the next by a space-like geodesic. We will see in the next section that those subsets $V_n$ can be considered as a ``discrete homeomorphism from $\RP^1$ to $\RP^1$'', that is, an increasing function from a finite subset of $\RP^1$ to $\RP^1$, and that it is ``uniformly quasi-symmetric'' in a natural sense. 

% With its Hamiltonian path, we connect all of its vertices by a path $v^n:\mathbb{RP}^1\rightarrow \partial_\infty$AdS$^3$ which is piecewise $\mathcal{C}^1$ with breaking points only at the vertices and with left and right derivatives everywhere. Moreover we take $v^n$ uniformly acausal in the sense that its (left and right) tangent are uniformly far from a lightlike line. This later condition rephrases as $v^n$ being quasi-symmetric. \marginnote{JM: je ne vois pas bien pourquoi, il me semble qu'il manque une preuve.} % ideal polyhedra and approximation -- JM

\section{Proof of Theorem \ref{tm:bending}.} %% JMS

We now come to the proof of the main theorem. 

\subsection{Discrete approximations of quasi-symmetric homeomorphisms}

We first introduce a simple notion of quasi-symmetry for maps from a finite subset of $\RP^1$ to $\RP^1$. In this section some definitions depend on a parameter $\varepsilon\in (0,1)$, however it appears possible to consider $\varepsilon$ fixed, for instance $\varepsilon=1/2$. 

\begin{defi}
  Let $F\subset \RP^1$ be a finite subset, and let $v:F\to \RP^1$ be an increasing map. Let $\varepsilon\in (0,1)$ and let $K>1$. We say that $v$ is $(K,\varepsilon)$-quasi-symmetric if for all $a,b,c,d\in F$ with $[a,b;c,d]\in [-1-\varepsilon,-1+\varepsilon]$, $[v(a),v(b); v(c),v(d)]\in [1/K,K]$.
\end{defi}

In the sequel, we say that  $a,b,c,d\in \mathbb{RP}^1$ with $[a,b;c,d]\in [-1-\varepsilon,-1+\varepsilon]$ are in $\varepsilon$-symmetric position.

\begin{lemma} \label{lm:restriction}
  For all $K>1$ and $\varepsilon \in (0,1)$, there exists $K'>0$ such that if $v:\RP^1\to \RP^1$ is a $K$-quasi-symmetric homeomorphism, then the restriction of $v$ to any finite subset of $\RP^1$ is $(K',\varepsilon)$-quasi-symmetric. 
\end{lemma}

\begin{proof}
We need to show that a quasi-symmetric map $u$ takes 4-tuples of points with bounded and bounded away from 0 cross-ratios to 4-tuples with the same property, with different constants. This fact is probably well-known, we provide a proof for the reader's convenience.

We consider a quasi-isometric extension of $u$. This is a map $U:\mathbb{H}^2\rightarrow\mathbb{H}^2$ which extends to $\partial_\infty\mathbb{H}^2=\mathbb{RP}^1$ as $u$ and satisfies, for all $x,y\in\mathbb{H}^2$,
\[\frac{1}{A}d(x,y)-A\leqslant d(U(x),U(y))\leqslant Ad(x,y)+A,\]
where $d$ is the hyperbolic distance and $A$ depends on the quasi-symmetric constant of $u$, see \cite{fletchermarkovic}.

For two points $\alpha$ and $\beta$ in $\mathbb{RP}^1$,  we denote by $(\alpha\beta)$ the hyperbolic geodesic with endpoints $\alpha$ and $\beta$.
Let $a,b,c,d$ be four points on $\mathbb{RP}^1$ such that 
\[1/K\leqslant cr(a;b;c;d)\leqslant K.\]
Those bounds on the cross-ratio translates into bounds on the distance between geodesics, namely, the cross-ratio is bounded and bounded away from 0 if and only if the distance between $(ab)$ and $(cd)$ and the distance between $(bc)$ and $(ad)$ are both bounded. This follows from the fact that a 4-tuple of points in $\mathbb{RP}^1$ is completely determined, modulo the isometry group of $\mathbb{H}^2$, by its cross-ratio.

By the quasi-isometry property, we get that the images $U(ab)$ and $U(cd)$ are within bounded distance, and same for the distances between $U(bc)$ and $U(ad)$. Moreover $U(ab)$ is within bounded distance from the geodesic $(u(a)u(b))$ \cite[paragraph 11.8, lemma 6]{ratcliffefoundations}. We conclude that the distance between $(u(a)u(b))$ and $(u(c)u(d))$ is bounded (same for $(u(b)u(c))$ and $(u(a)u(d))$), which in turn gives that the cross ratio
\[cr(u(a);u(b);u(c);u(d))\]
is bounded and bounded away from 0. %%\marginnote{L: Recently added}
\end{proof}

Conversely, if a sequence of discrete uniformly quasi-symmetric homeomorphisms converges to a limit, it must be quasi-symmetric.

\begin{lemma} \label{lem:limit}
  Let $u:\RP^1\to \RP^1$ be a continuous map, and let $(F_n,v_n)_{n\in \N}$ be a sequence of pairs such that:
  \begin{itemize}
  \item for all $n$, $F_n\subset \RP^1$ is a finite subset, and for all $x\in \RP^1$, there exists a sequence $(x_n)_{n\in \N}$, with $x_n\in F_n$, such that $x_n\to x$,
  \item $v_n\to u$, in the sense that if $x_n\in F_n$ and $x_n\to x$ then $v_n(x_n)\to u(x)$,
  \item the $v_n$ are $(K,\varepsilon)$-quasi-symmetric, for fixed $K>1$ and $\varepsilon>0$.  
  \end{itemize}
  Then $u$ is $K$-quasi-symmetric.
\end{lemma}

\begin{proof}
Let $a,b,c,d$ be 4 points in $\mathbb{RP}^1$ in symmetric position. There exists 4 sequences $(a_n), (b_n), (c_n)$ and $(d_n)$ converging to $a,b,c$ and $d$ respectively. Since the cross-ratio is a continuous function, for $n$ big enough, the 4 points $a_n,b_n,c_n$ and $d_n$ are in $\varepsilon$-symmetric position. It follows that $[v(a_n),v(b_n); v(c_n),v(d_n)]\in [1/K,K]$. The result follows from the definition of $u$ and continuity of the cross-ratio.
\end{proof}

Finally we can state a kind of compactness result for $(K,\varepsilon)$-quasi-symmetric maps.

\begin{lemma}\label{lem:precompactness}
  Let $K>1, \varepsilon\in (0,1)$, and let $(F_n,v_n)_{n\in \N}$ be a sequence of pairs, where
  \begin{enumerate}
  \item for each $n$, $F_n$ is a finite subset of $\RP^1$, and for all $x\in \RP^1$, there exists a sequence $(x_n)_{n\in \N}$, with $x_n\in F_n$, such that $x_n\to x$,
  \item the $v_n$ are $(K,\varepsilon)$-quasi-symmetric.
  \end{enumerate}
  Then there exists a subsequence of $(F_n, v_n)_{n\in \N}$ converging to a limit $u$ in the sense of Lemma \ref{lem:limit}.
\end{lemma}

This limit $u$ is then quasi-symmetric by Lemma \ref{lem:limit}.

\begin{proof}
Without loss of generality, we can assume that $n$ is large enough so that $F_n$ has at least three points. Since pre- and post-composing by a Möbius transformation does not change the fact that $v_n$ is $(K,\varepsilon)$-quasi-symmetric, we can assume that the points $0,-1,\infty\in\RP^1$ are sent to $0,-1,\infty$ respectively.

Since $\mathbb{RP}^1$ is compact,$(F_n,v_n)_{n\in\N}$ has a converging subsequence. Indeed, the Arzela-Ascoli Theorem applies to the set of maps $v_n:F_n\rightarrow \mathbb{RP}^1$ for a fixed $n$ and implies that the set of all sequences $(F_n,v_n)_{n\in \N}$ is relatively compact by Tychonoff's theorem. We denote by $u$ this sublimit. It is defined on the whole $\RP^1$ because of hypothesis (1). %%\marginnote{JM: je ne comprends pas bien pourquoi on a besoin de la suite, en fait. On sait que la limite est quasisymétrique par le lemme 4.3 et non constante donc c'est un homéo, non?}

It remains to show that a non-constant and non-decreasing quasi-symmetric map from $\mathbb{RP}^1$ onto itself is a homeomorphism. This is certainly well-known to experts but we include a proof for completeness.

Note that $u$ also satisfies $u(-1)=-1$, $u(0)=0$ and $u(\infty)=\infty$. %By lemma \ref{lem:limit}, $u$ is $K$-quasi-symmetric (i.e the images of 4 points in symmetric position have a bounded cross-ratio) and non-decreasing.
We first show that $u$ is continuous. By contradiction, assume that $u$ has a ``jump'' at $x\in\mathbb{RP}^1$: there exists a sequence $y_n$ converging to $x$ such that, for a fixed $\delta$, $\abs{u(x)-u(y_n)}>\delta$. Take any point $z\in\mathbb{RP}^1$, $z\neq x$. There exists a sequence $(t_n)$ such that $[x,y_n;z,t_n]=-1$ (in particular $t_n\to z$) and $[u(x),u(y_n);u(z),u(t_n)]$ is bounded. Hence $u(t_n)$ cannot approach $u(z)$. Since $u$ is non-decreasing, what is true for $(t_n)$ is also true in a whole neighborhood of $z$. We conclude that $u(z)$ is isolated in $u(\mathbb{RP}^1)$. Since $z$ was chosen arbitrarily (different from $x$), this means that $u$ has at most 2 values, contradicting the fact that $u(-1)=-1$, $u(0)=0$ and $u(\infty)=\infty$.

Finally, we show that $u$ is injective using the same kind of argument. Assuming that $u(x)=u(y)$, we take any $z\in\mathbb{RP}^1$ and find the point $t\in\mathbb{RP}^1$ such that $\crochet{x,y;z,t}=-1$. For $\crochet{u(x),u(y);u(z),u(t)}$ to remain bounded, we must have $u(z)=u(t)$ and then again, $u$ has only 2 values.
\end{proof}

\subsection{The discrete approximations are uniformly quasi-symmetric.}

Recall that we started with a pair $(\lambda_-, \lambda_+)$  of laminations on $\HH^2$ that strongly fill. We associated to this pair:
\begin{itemize}
\item Sequences $\lambda^n_\pm=\lambda^n_\pm$ of polyhedral pairs converging to $\lambda_\pm$, with $k_n$ vertices. %%By construction, those 
\item A sequence of quasi-symmetric embeddings $v^n:\Z/k_n\Z\to\partial AdS^3$ such that $v^n_*(\lambda^n_\pm)$ are the measured bending laminations on the boundary of $CH(v^n(Z/k_n\Z))$.
\item A sequence of ideal polyhedra $P^n=CH(v^n(\Z/k_n\Z))\subset AdS^3$. 
\item A sequence of pairs of ideal polygons $p^n_\pm\subset \HH^2$, isometric respectively to the future and past boundary components of $P^n$.
\end{itemize}
We denote by $V^n$ the set of (ideal) vertices of $P^n$. The elements of $V^n$ are also vertices of both $p^n_-$ and $p^n_+$, which can be considered as points of $\RP^1$ since $p^n_\pm$ are ideal hyperbolic polygons. Therefore $V^n$ can be identified with subsets $V^n_\pm$ of $\RP^1$. We will use this identification below. 

We moreover have a sequence of pairs of convex isometric embeddings $U^n_\pm:p^n_\pm\to AdS^3$, with images the upper and lower boundary components of $CH(v^n(\Z/k_n\Z))$. The boundary maps $u^n_\pm:V^n_\pm\subset \mathbb{RP}^1\rightarrow \partial AdS^3$ realize the identification of $V^n_\pm$ with $V^n$, the set of vertices of the ideal polyhedron $P^n$.

We also denote by $\lambdab^n_\pm$ the measured bending laminations on the future and past boundary components of $P^n$. Note that those laminations are related to $\lambda^n_\pm$, but not in a completely straightforward way -- for instance, $\lambdab^n\pm$ is not the image of $\lambda^n_\pm$ by an isometric embedding of $\HH^2$ into $AdS^3$. A key point of the arguments below is that the measured bending laminations $\lambdab^n_\pm$ are uniformly bounded, as measured laminations on the ideal polygons $p^n_\pm$. 

The following lemma is useful in this respect. We define the {\em width} of an ideal polyhedron in $AdS^3$ just as for convex hulls of acausal meridians, that is, as the supremum of the time distance between the past and future boundary components. 

\begin{lemma} \label{lm:w}
  There exists $w_0<\pi/2$ such that for all $n$, the width of $P_n$ is at most $w_0$.
\end{lemma}

\begin{proof}
  We argue by contradiction and assume that the widths of the polyhedra $P_n$ is not bounded away from $\pi/2$. Taking a subsequence, we can then suppose that $w(P_n)\to \pi/2$. There are then two sequences of points $(x_-^n)_{n\in \N}$ and $(x_+^n)_{n\in \N}$, with $x_-^n\in \partial_-P_n, x_+^n\in \partial_+P_n$, and such that the time distance from $x_-^n$ to $x_+^n$ goes to $\pi/2$ as $n\to \infty$.  After applying a sequence of isometries so that $x_-^n\to x_-$ and $x_+^n\to x_+$ in the projective model of $AdS^3$, we see as in \cite[Lemma 8.4]{convexhull} that $(P_n)_{n\in \N}$ converges, in the Hausdorff topology on compact subsets of $AdS^3$, to a rhombus, with upper axis containing $y$ and lower axis containing $x$. %% \marginnote{JM: check defs given. L: I added something in chapter 2. The axis are the pleating lines, right ?}

  In this case, $x_-^n$ and $x_+^n$ can be chosen to be on pleating lines $d_-^n$ and $d_+^n$ of $P_n$. Because $P_n$ converges to a rhombus, the intersection with $\lambda^n_\pm$ of a geodesic segment of fixed length orthogonal to $d^n_\pm$ goes to infinity as $n\to \infty$. 

  We can now apply Lemma \ref{lem:reciprocalbounds}, with $a,b,c,d$ chosen to be four vertices of $P_n$ such that the endpoints of $d^n_+$ are in $(a,b)$ and $(c,d)$, respectively, while the endpoints of $d^n_-$ are $(b,c)$ and $(d,a)$, respectively, see figure \ref{fig:lemma 4.5}. This lemma leads to a contradiction, which shows that the widths of the polyhedra $P_n$ are uniformly bounded away from $\pi/2$.

\definecolor{xdxdff}{rgb}{0.49019607843137253,0.49019607843137253,1}
\definecolor{ttqqqq}{rgb}{0.2,0,0}
\definecolor{qqqqff}{rgb}{0,0,1}
\begin{figure}[!h]
\begin{tikzpicture}[line cap=round,line join=round,>=triangle 45,x=1cm,y=1cm, scale=0.5]
\clip(-13.518102549015099,-9.606661660539363) rectangle (22.939726381577152,11.933364221304464);
\draw [shift={(-0.6120664456333484,-12.611420067468222)},line width=2pt]  plot[domain=1.2076900773290586:1.8924404522798597,variable=\t]({1*22.25386773314934*cos(\t r)+0*22.25386773314934*sin(\t r)},{0*22.25386773314934*cos(\t r)+1*22.25386773314934*sin(\t r)});
\draw [shift={(0.459184531489201,39.05489142685662)},line width=2pt]  plot[domain=4.453050696185044:4.93026514060346,variable=\t]({1*31.610757505913618*cos(\t r)+0*31.610757505913618*sin(\t r)},{0*31.610757505913618*cos(\t r)+1*31.610757505913618*sin(\t r)});
\draw [shift={(-34.46349617887165,-0.02446447903903247)},line width=2pt]  plot[domain=-0.273844492272322:0.30782180563791295,variable=\t]({1*28.13903919524748*cos(\t r)+0*28.13903919524748*sin(\t r)},{0*28.13903919524748*cos(\t r)+1*28.13903919524748*sin(\t r)});
\draw [shift={(32.63292579975639,0.5648293708129372)},line width=2pt]  plot[domain=2.849254115239226:3.444032116158506,variable=\t]({1*26.463656698355372*cos(\t r)+0*26.463656698355372*sin(\t r)},{0*26.463656698355372*cos(\t r)+1*26.463656698355372*sin(\t r)});
\draw [shift={(-0.4385501511652312,12.869019450215317)},line width=2pt]  plot[domain=4.38625549993096:5.081500630402397,variable=\t]({1*21.644158421451422*cos(\t r)+0*21.644158421451422*sin(\t r)},{0*21.644158421451422*cos(\t r)+1*21.644158421451422*sin(\t r)});
\draw [line width=2.8pt,color=qqqqff] (-6.55132273164352,3.541499665975146)-- (6.692861451633188,5.8029767020395635);
\draw [line width=0.8pt,color=qqqqff] (6.692861451633188,5.8029767020395635)-- (5.209258796191471,-4.088142632959539);
\draw [line width=0.8pt,dash pattern=on 1pt off 1pt,color=qqqqff] (-6.55132273164352,3.541499665975146)-- (5.209258796191471,-4.088142632959539);
\draw [line width=2.8pt,color=qqqqff] (5.209258796191471,-4.088142632959539)-- (-5.934589304663387,-3.5540924044840727);
\draw [line width=0.8pt,color=qqqqff] (-5.934589304663387,-3.5540924044840727)-- (6.692861451633188,5.8029767020395635);
\draw [line width=0.8pt,color=qqqqff] (-6.55132273164352,3.541499665975146)-- (-5.934589304663387,-3.5540924044840727);
\draw [color=ttqqqq](-1.0569305512540763,6) node[anchor=north west] {$d_+^n$};
\draw (-2.5166678424203672,-3.7) node[anchor=north west] {$d_-^n$};
\draw (-9.566130858296603,-5.654689969820909) node[anchor=north west] {$\partial AdS^3$};
\draw [line width=2pt] (14.750956211619908,0.8607228175797859) circle (6.045863734809703cm);
\draw [line width=2.8pt,color=qqqqff] (12.221970590605281,6.352234451782921)-- (16.566443694764956,-4.906119775939719);
\draw [line width=2.8pt,color=qqqqff] (9.162721063801877,-1.446677501519247)-- (17.16679521446377,6.40294170645683);
\draw (11.6,5) node[anchor=north west] {$d_-^n$};
\draw (15.391816485790473,5) node[anchor=north west] {$d_+^n$};
\draw (18.34689441668809,-3.4116790102239483) node[anchor=north west] {$\mathbb{RP}^1$};
\begin{scriptsize}
\draw [fill=xdxdff] (14.931889019002421,6.903878584155707) circle (2.5pt);
\draw[color=xdxdff] (15.213799742965318,7.678764067783245) node {$a$};
\draw [fill=xdxdff] (20.35456279111948,3.1305381409213613) circle (2.5pt);
\draw[color=xdxdff] (20.625508724850103,3.9048091198899466) node {$b$};
\draw [fill=xdxdff] (12.78556822825861,-4.856769497653015) circle (2.5pt);
\draw[color=xdxdff] (13.077598829063428,-4.105944307242055) node {$c$};
\draw [fill=xdxdff] (9.357638737958542,3.5928639062108436) circle (2.5pt);
\draw[color=xdxdff] (9.0900237897799,4.296445954105289) node {$d$};
\end{scriptsize}
\end{tikzpicture}
\label{fig:lemma 4.5}
\caption{Lemma \ref{lm:w}}
\end{figure}

\end{proof}

\begin{lemma} \label{lm:bounded}
  There exists $k_0>0$ such that for all $n$ and all geodesic segments $\gamma$ of length $1$ in $p^n_\pm$, the intersection of $\gamma$ with $\lambdab^n_\pm$ is at most $k_0$.
\end{lemma}

\begin{proof}
  We again argue by contradiction and suppose that there is a sequence of geodesic segments $\gamma^n$ in $p^n_+$ such that $i(\gamma^n, \lambdab^n_+)\to \infty$. There is then a sequence of points $x^n\in \gamma^n$ and a sequence $\varepsilon^n\to 0$ such that the intersection of $\lambdab^n_+$ with the segment of $\gamma^n$ of length $2\varepsilon^n$ centered at $x_n$ goes to $\infty$.

  This simplies that, after taking a subsequence and normalization by a sequence of isometries, $p^n_+$ converges (in the Hausdorff topology on compact subsets of $AdS^3$) to the union between two lightlike half-planes $q$ and $q'$ intersecting along a line containing $x^n_+$.

  A direct convexity argument then shows that $P^n$ converges to the convex hull of $q\cup q'$, which is a rhombus. As a consequence, $w(P_n)\to \pi/2$, contradicting Lemma \ref{lm:w}. 
\end{proof}

As a consequence we can now see that the maps $u^n_\pm:V^n_\pm\to \partial AdS^3$ are in a natural way ``parameterized discrete quasi-circles''. Recall that $\pi_L,\pi_R:\partial AdS^3\to \RP^1$ denote the left and right projections on the boundary of $AdS^3$.

\begin{lemma} \label{lm:qsym}
  The maps $\pi_L\circ u^n_\pm:V^n_\pm\to \RP^1$ and $\pi_R\circ u^n_\pm:V^n_\pm \to\RP^1$ are $(K,\varepsilon)$-quasi-symmetric, for fixed $K>1$ and $\varepsilon>0$. 
\end{lemma}

\begin{proof}
  This follows from Lemma \ref{lm:restriction} since $\pi_L\circ u^n_\pm$ is the restriction to $V^n_\pm$ of the earthquake along $\lambdab^n_\pm$, which is bounded, and the boundary map of the earthquake along a (uniformly) bounded measured lamination is (uniformly) quasi-symmetric. 
\end{proof}

\subsection{Convergence of bending laminations}

We are now equipped to prove Theorem \ref{tm:bending}

Lemma \ref{lm:qsym} indicates that the maps $\pi_L\circ u^n_\pm$ are $(K,\varepsilon)$-quasi-symmetric, for some $K>1$ and $\varepsilon>0$. It then follows from Lemma \ref{lem:precompactness} that after extracting a subsequence, $\pi_L\circ u^n_\pm\to u'_{L,\pm}$, where the maps $u'_{L,\pm}:\RP^1\to \RP^1$ are $K$-quasi-symmetric. The same argument shows that after again extracting a subsequence, $\pi_R\circ u^n_{R,\pm}\to u'_{R,\pm}$, where again the maps $u'_{R,\pm}$ are quasi-symmetric. It follows that $u^n_\pm\to u_\pm$, where the maps $u_\pm:\RP^1\to \RP^1\times \RP^1=\partial AdS^3$ are parameterized quasicircles. Moreover, for all $n$, $u^n_+(V^n)=u^n_-(V^n)$, and therefore $u_+$ and $u_-$ have the same image and are therefore two parameterization of the same quasicircle.

For each $n$, the measured laminations $\lambdab^n_\pm$ are the measured bending laminations on the past and future boundary of $P_n$. As $n\to \infty$, after possibly extracting a subsequence, $P_n\to CH(u^+(\RP^1))$, in the Hausdorff topology on subsets of $AdS^3$. Then $\lambdab^n_\pm\to \lambdab_\pm$, where the $\lambdab_\pm$ are the measured bending laminations on the past and future boundary components of $CH(u_+(\mathbb{RP}^1))$.

But $\pi_L\circ u^n_\pm\to \pi_L\circ u_\pm$ as quasi-symmetric homeomorphisms, and $\lambda^n_\pm\to \lambda_\pm$, and it follows that
$$ (\pi_L\circ u^n_\pm)_*(\lambda^n_\pm) \to (\pi_L\circ u_\pm)_*(\lambda_\pm)~. $$
The same holds with$\pi_L$ replaced by $\pi_R$, and therefore 
$$ (u^n_\pm)_*(\lambda^n_\pm) \to (u_\pm)_*(\lambda_\pm)~. $$

However for $n\in \N$, $(u^n_\pm)_*(\lambda^n_\pm)=\lambdab^n_\pm$ by construction of $P_n$, and therefore
$$ (u^n_\pm)_*(\lambda^n_\pm)=\lambdab^n_\pm \to \lambdab_\pm~. $$

Comparing the last two equalities, we obtain that $\lambdab_\pm=(u_\pm)_*(\lambda_\pm)$, which concludes the proof of Theorem \ref{tm:bending}.

% proof of main theorem -- JM

\section{Earthquakes} \label{sc:earthquakes} %% LM

In this section we indicate how to prove Theorem \ref{tm:earthquakes} from Theorem \ref{tm:bending}.

The main tool in the proof is the following relation discovered by G. Mess between earthquakes and quasicircles in the asymptotic boundary of $AdS^3$, see \cite[Proposition 22]{mess}, \cite{mess-notes}. %Here homeomorphisms from $\RP^1$ to $\R P^1$ are considered up to pre- and post-composition by Möbius transformations.

\begin{prop} \label{pr:mess22}
  Let $\phi:\RP^1\to \RP^1$ be an orientation-preserving homeomophism, and let $\lambda_+$ (resp. $\lambda_-$) be the measured bending lamination on the future (resp. past) boundary component of the convex hull of the graph of $\phi$ in $\RP^1\times \RP^1$, identified with $\partial AdS^3$. Then the left earthquake along $2\lambda_+$ (resp. the right earthquake along $2\lambda_-$) extends continuously to $\RP^1$, identified with $\partial_\infty \HH^2$, and its restriction to $\RP^1$ equals $\phi$ (up to pre- and post-composition by a Möbius transformation).
\end{prop}

We will also use \cite[Proposition 23]{mess}, which we can reformulate using our notations as follows.

%% definition of left and right projections from boundary to RP^1, well-defined up to left composition by Möbius

%% def of embedding u_\pm of HH^2 in upper and lower boundary components

%% choose normalisation so that 

\begin{prop} \label{pr:mess23}
  With the notations above,
  $$ u_l = \cE^l(\lambda_+)(u_+) = \cE^r(\lambda_-)(u_-)~, $$
  $$ u_r = \cE^r(\lambda_+)(u_+) = \cE^l(\lambda_-)(u_-)~. $$
\end{prop}

\begin{proof}[Proof of Theorem \ref{tm:earthquakes}]
  Let $\lambda_l, \lambda_r\in \cML$ two bounded measured laminations that strongly fill. Let $\lambda_+=\lambda_l/2$, $\lambda_-=\lambda_r/2$. By Theorem \ref{tm:bending}, there exists a parameterized quasicircle $u:\RP^1\to \partial AdS^3$ such that $u_*(\lambda_\pm)$ are the measured bending laminations on the past and future boundary components of the convex hull of $u(\RP^1)$.

  Let $u_\pm:\RP^1\to \partial AdS^3$ be the parameterization of $u(\RP^1)$ obtained as boundary values of isometries from $\HH^2$ to $\partial_\pm CH(u(\RP^1)$, and let $u_l=\pi_l\circ u$, $u_r = \pi_r \circ u$.

  By Proposition \ref{pr:mess23}, we have 
  $$ u_l = \cE^l(\lambda_+)(u_+) = \cE^r(\lambda_-)(u_-)~, $$
  $$ u_r = \cE^r(\lambda_+)(u_+) = \cE^l(\lambda_-)(u_-)~. $$
  As a consequence,
  $$ u_+=\cE^l(\lambda_+)(u_r)~, ~~ u_- = \cE^r(\lambda_-)(u_-)~, $$
  and therefore
  $$ u_l = \cE^l(\lambda_+)(\cE^l(\lambda_+)(u_r) = \cE^l(2\lambda_+)(u_r) = \cE^l(\lambda_l)(u_r)~, $$
  $$ u_l = \cE^r(\lambda_-)(\cE^r(\lambda_-)(u_r) = \cE^r(2\lambda_-)(u_r) = \cE^r(\lambda_r)(u_r)~, $$
  which is precisely the statement of the theorem.
\end{proof} % relation to earthquakes -- LM

\section{Examples}
\label{sc:examples}

\subsection{Necessary conditions on pairs laminations} \label{ssc:necessary}

As mentioned in the introduction, the condition in Theorem \ref{tm:bending} that $\lambda_+$ and $\lambda_-$ strongly fill is certainly not necessary. Whenever $\Gamma\subset \partial AdS^3$ is a smooth, uniformly space-like curve, the boundary of its convex hull is asymptotically ``flat'' near $\Gamma$, so that the future and past bending laminations cannot strongly fill in the sense of Definition \ref{df:fill}.

However, it is also clear that a much weaker filling condition must always be satisfied.

\begin{remark} \label{rk:weak}
Let $\Gamma$ be a closed acausal meridian in $\partial AdS^3$, which is not the boundary of a plane. Let $\lambda_+$ and $\lambda_-$ be the measured bending laminations on the future and past boundary components of $CH(\Gamma)$, considered as measures on $\Gamma\times \Gamma\setminus \Delta$, where $\Delta\subset \Gamma\times \Gamma$ is the diagonal. Then for any complete geodesic $\gamma\subset \HH^2$, $i(\lambda_+,\gamma)+i(\lambda_-,\gamma)>0$.   
\end{remark}

Another perhaps clearer way to state this remark, in terms of measures, is that if $a,b\in \Lambda$ are distinct, then $\lambda_+((a,b)\times (b,a))+\lambda_-((a,b)\times (b,a))>0$.

\begin{proof}
Suppose that this is not the case. Then $\gamma$ would be realized as a geodesic of $AdS^3$ (without bending) in both the future and past boundary components of $CH(\Gamma)$. Therefore those two boundary components $\partial_+CH(\Gamma)$ and $\partial_-CH(\Gamma)$ would intersect along a complete geodesic, that we can still denote $\gamma$. But then if $P_-$ is a support plane of $CH(\Gamma)$ along $\gamma$ in the past, and $P_+$ is a support plane of $CH(\Gamma)$ along $\gamma$ in the future, then $CH(\Gamma)$ is in the future of $P_-$ and in the past of $P_+$, while both $P_-$ and $P_+$ contain $\gamma$. This can only happen if $P_+=\partial_+CH(\Gamma)=\partial_-CH(\Gamma)=P_-$, that is, if $\Gamma$ bounds a plane.
\end{proof}

\subsection{An example of a pair that cannot be realized} \label{ssc:example}

Start by considering a complete, space-like, future-convex surface $\Sigma_-\subset AdS^3$ which is the union of two totally geodesic half-planes meeting along their common boundary. Let $C=CH(\Sigma_-)$ be its convex core. By construction, $\Sigma_-$ is the past boundary component of $C$. A simple symmetry argument shows that the upper boundary component of $C$, $\partial_+C$, is bent along a measured foliation which is invariant under a one-parameter subgroup of translations --- each leaf of the foliation is orthogonal to the common axis of the translations. We denote by $\lambda_+, \lambda_-$ the measured bending laminations on the future and past boundary components of $C$, so that the support of $\lambda_-$ contains only one line, while the support of $\lambda_+$ is the whole hyperbolic plane, figure see \ref{fig:weakfilling}.

\definecolor{qqqqff}{rgb}{0,0,1}
\definecolor{ccqqqq}{rgb}{0.8,0,0}
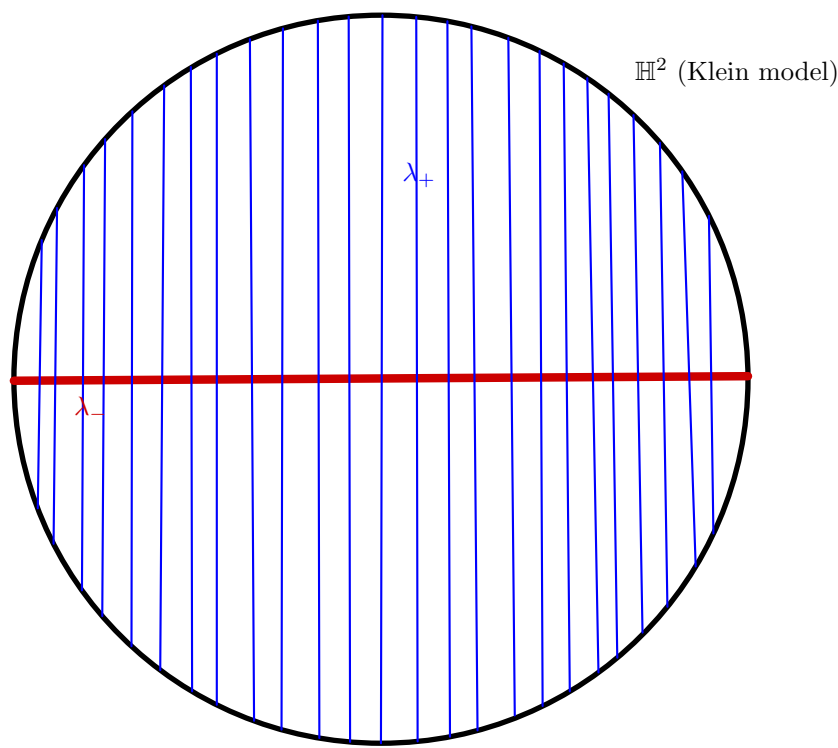
\begin{figure}[!h]
\begin{tikzpicture}[line cap=round,line join=round,>=triangle 45,x=1cm,y=1cm]
\clip(-4.3,-5.8) rectangle (16.18,6.3);
\draw [line width=2pt] (4.58,0.36) circle (4.827131653477041cm);
\draw [line width=3.2pt,color=ccqqqq] (-0.24708975183814275,0.33988712603400845)-- (9.406965441382319,0.4000578044927994);
\draw [line width=0.8pt,color=qqqqff] (0.675254250818278,3.197985312549602)-- (0.6479800369149946,-2.4400748221968978);
\draw [line width=0.8pt,color=qqqqff] (1.3108847497528004,3.911631382984611)-- (1.2992114301301103,-3.1808510784571378);
\draw [line width=0.8pt,color=qqqqff] (2.08,4.5)-- (2.0964612533962055,-3.7792312443396594);
\draw [line width=0.8pt,color=qqqqff] (2.8565646757213536,4.868987767008084)-- (2.9112347196278296,-4.169505761010177);
\draw [line width=0.8pt,color=qqqqff] (3.750351808702698,5.115300608655263)-- (3.7700104556029403,-4.3986885733327235);
\draw [line width=0.8pt,color=qqqqff] (4.600112873965992,5.187089751838143)-- (4.58,-4.467131653477041);
\draw [line width=0.8pt,color=qqqqff] (5.4539893275088005,5.107351119877348)-- (5.487849611461375,-4.380992415409397);
\draw [line width=0.8pt,color=qqqqff] (6.255232973878709,4.887117679410321)-- (6.34,-4.1);
\draw [line width=0.8pt,color=qqqqff] (6.983551030776951,4.546184711936521)-- (7.063538746603795,-3.7792312443396594);
\draw [line width=0.8pt,color=qqqqff] (7.573675116882556,4.146701637911048)-- (7.688084872981961,-3.3333735828292133);
\draw [line width=0.8pt,color=qqqqff] (8.249205801804088,3.4965791531551074)-- (8.347379666780826,-2.657954679302919);
\draw [line width=0.8pt,color=qqqqff] (8.893497791045313,2.526780285734389)-- (8.952213593314664,-1.6857146170554853);
\draw [line width=0.8pt,color=qqqqff] (0.11940212932012617,2.2050654828721297)-- (0.06687426685145326,-1.3525700326679746);
\draw [line width=0.8pt,color=qqqqff] (3.2911640410682095,5.011892289269429)-- (3.2724929357655617,-4.28667895135623);
\draw [line width=0.8pt,color=qqqqff] (4.155725006173748,5.168449930030855)-- (4.17408987108418,-4.4500350276524525);
\draw [line width=0.8pt,color=qqqqff] (5.0443243817775745,5.1647479505679454)-- (5.06230722679509,-4.442976133501104);
\draw [line width=0.8pt,color=qqqqff] (5.850102358079852,-4.2970419796261226)-- (5.7697203685946175,5.0382224663381585);
\draw [line width=0.8pt,color=qqqqff] (6.666240874833221,4.713021825373164)-- (6.706297392047093,-3.973596589505503);
\draw [line width=0.8pt,color=qqqqff] (7.4407003416180215,-3.5281349713540715)-- (7.293806588574286,4.3520488223230425);
\draw [line width=0.8pt,color=qqqqff] (7.901508660835276,3.8626818605172004)-- (8.02,-3);
\draw [line width=0.8pt,color=qqqqff] (8.543027466004697,3.116014024581947)-- (8.71923124433966,-2.1235387466037956);
\draw [line width=0.8pt,color=qqqqff] (2.4212410972968756,4.677517805406249)-- (2.4212410972968756,-3.957517805406249);
\draw [line width=0.8pt,color=qqqqff] (1.728892654875021,4.255174823621449)-- (1.6772293948441033,-3.496828076780563);
\draw [line width=0.8pt,color=qqqqff] (0.9522117042810803,3.544391948464386)-- (0.9149684216110452,-2.781455638619103);
\draw [line width=0.8pt,color=qqqqff] (0.32021009852853055,2.6305483908786607)-- (0.274452584628337,-1.8225355561791847);
\draw (7.80,4.74) node[anchor=north west] {$\mathbb{H}^2$ (Klein model)};
\draw [color=ccqqqq](0.42,0.25) node[anchor=north west] {$\lambda_-$};
\draw [color=qqqqff](4.74,3.36) node[anchor=north west] {$\lambda_+$};
\end{tikzpicture}
\label{fig:weakfilling}
\caption{The lower lamination (in red) and the upper (in blue) intersect every bi-infinite geodesic but the pair does not strongly fill.}
\end{figure}

Clearly, the pair $(\lambda_-, \lambda_+)$ does not strongly fill, since a long segment of a bending line of $\lambda_+$ can be disjoint from the support of $\lambda_-$.

Now consider a modified lamination $\lambda_+'$ obtained from $\lambda_+$ by adding an atomic weight to one leaf of the foliation. It is then clear that the pair $(\lambda_-,\lambda_+')$ weakly fills, in the sense of Remark \ref{rk:weak}. However it cannot be realized as measured bending laminations on the past and future boundary components of the convex hull of a curve, since in this case $\lambda_-$ by itself determines the curve and therefore the upper bending lamination, which needs to be equal to $\lambda_+$. % examples LM

\bibliographystyle{alpha}
\bibliography{biblio}

\end{document}